\documentclass[12pt]{article}

\usepackage[utf8]{inputenc}

\usepackage{color} 
\definecolor{AliceBlue}{rgb}{0.94,0.97,1.00}
\definecolor{AntiqueWhite1}{rgb}{1.00,0.94,0.86}
\definecolor{AntiqueWhite2}{rgb}{0.93,0.87,0.80}
\definecolor{AntiqueWhite3}{rgb}{0.80,0.75,0.69}
\definecolor{AntiqueWhite4}{rgb}{0.55,0.51,0.47}
\definecolor{AntiqueWhite}{rgb}{0.98,0.92,0.84}
\definecolor{BlanchedAlmond}{rgb}{1.00,0.92,0.80}
\definecolor{BlueViolet}{rgb}{0.54,0.17,0.89}
\definecolor{CadetBlue1}{rgb}{0.60,0.96,1.00}
\definecolor{CadetBlue2}{rgb}{0.56,0.90,0.93}
\definecolor{CadetBlue3}{rgb}{0.48,0.77,0.80}
\definecolor{CadetBlue4}{rgb}{0.33,0.53,0.55}
\definecolor{CadetBlue}{rgb}{0.37,0.62,0.63}
\definecolor{CornflowerBlue}{rgb}{0.39,0.58,0.93}
\definecolor{DarkBlue}{rgb}{0.00,0.00,0.55}
\definecolor{DarkCyan}{rgb}{0.00,0.55,0.55}
\definecolor{DarkGoldenrod1}{rgb}{1.00,0.73,0.06}
\definecolor{DarkGoldenrod2}{rgb}{0.93,0.68,0.05}
\definecolor{DarkGoldenrod3}{rgb}{0.80,0.58,0.05}
\definecolor{DarkGoldenrod4}{rgb}{0.55,0.40,0.03}
\definecolor{DarkGoldenrod}{rgb}{0.72,0.53,0.04}
\definecolor{DarkGray}{rgb}{0.66,0.66,0.66}
\definecolor{DarkGreen}{rgb}{0.00,0.39,0.00}
\definecolor{DarkGrey}{rgb}{0.66,0.66,0.66}
\definecolor{DarkKhaki}{rgb}{0.74,0.72,0.42}
\definecolor{DarkMagenta}{rgb}{0.55,0.00,0.55}
\definecolor{DarkOliveGreen1}{rgb}{0.79,1.00,0.44}
\definecolor{DarkOliveGreen2}{rgb}{0.74,0.93,0.41}
\definecolor{DarkOliveGreen3}{rgb}{0.64,0.80,0.35}
\definecolor{DarkOliveGreen4}{rgb}{0.43,0.55,0.24}
\definecolor{DarkOliveGreen}{rgb}{0.33,0.42,0.18}
\definecolor{DarkOrange1}{rgb}{1.00,0.50,0.00}
\definecolor{DarkOrange2}{rgb}{0.93,0.46,0.00}
\definecolor{DarkOrange3}{rgb}{0.80,0.40,0.00}
\definecolor{DarkOrange4}{rgb}{0.55,0.27,0.00}
\definecolor{DarkOrange}{rgb}{1.00,0.55,0.00}
\definecolor{DarkOrchid1}{rgb}{0.75,0.24,1.00}
\definecolor{DarkOrchid2}{rgb}{0.70,0.23,0.93}
\definecolor{DarkOrchid3}{rgb}{0.60,0.20,0.80}
\definecolor{DarkOrchid4}{rgb}{0.41,0.13,0.55}
\definecolor{DarkOrchid}{rgb}{0.60,0.20,0.80}
\definecolor{DarkRed}{rgb}{0.55,0.00,0.00}
\definecolor{DarkSalmon}{rgb}{0.91,0.59,0.48}
\definecolor{DarkSeaGreen1}{rgb}{0.76,1.00,0.76}
\definecolor{DarkSeaGreen2}{rgb}{0.71,0.93,0.71}
\definecolor{DarkSeaGreen3}{rgb}{0.61,0.80,0.61}
\definecolor{DarkSeaGreen4}{rgb}{0.41,0.55,0.41}
\definecolor{DarkSeaGreen}{rgb}{0.56,0.74,0.56}
\definecolor{DarkSlateBlue}{rgb}{0.28,0.24,0.55}
\definecolor{DarkSlateGray1}{rgb}{0.59,1.00,1.00}
\definecolor{DarkSlateGray2}{rgb}{0.55,0.93,0.93}
\definecolor{DarkSlateGray3}{rgb}{0.47,0.80,0.80}
\definecolor{DarkSlateGray4}{rgb}{0.32,0.55,0.55}
\definecolor{DarkSlateGray}{rgb}{0.18,0.31,0.31}
\definecolor{DarkSlateGrey}{rgb}{0.18,0.31,0.31}
\definecolor{DarkTurquoise}{rgb}{0.00,0.81,0.82}
\definecolor{DarkViolet}{rgb}{0.58,0.00,0.83}
\definecolor{DeepPink1}{rgb}{1.00,0.08,0.58}
\definecolor{DeepPink2}{rgb}{0.93,0.07,0.54}
\definecolor{DeepPink3}{rgb}{0.80,0.06,0.46}
\definecolor{DeepPink4}{rgb}{0.55,0.04,0.31}
\definecolor{DeepPink}{rgb}{1.00,0.08,0.58}
\definecolor{DeepSkyBlue1}{rgb}{0.00,0.75,1.00}
\definecolor{DeepSkyBlue2}{rgb}{0.00,0.70,0.93}
\definecolor{DeepSkyBlue3}{rgb}{0.00,0.60,0.80}
\definecolor{DeepSkyBlue4}{rgb}{0.00,0.41,0.55}
\definecolor{DeepSkyBlue}{rgb}{0.00,0.75,1.00}
\definecolor{DimGray}{rgb}{0.41,0.41,0.41}
\definecolor{DimGrey}{rgb}{0.41,0.41,0.41}
\definecolor{DodgerBlue1}{rgb}{0.12,0.56,1.00}
\definecolor{DodgerBlue2}{rgb}{0.11,0.53,0.93}
\definecolor{DodgerBlue3}{rgb}{0.09,0.45,0.80}
\definecolor{DodgerBlue4}{rgb}{0.06,0.31,0.55}
\definecolor{DodgerBlue}{rgb}{0.12,0.56,1.00}
\definecolor{FloralWhite}{rgb}{1.00,0.98,0.94}
\definecolor{ForestGreen}{rgb}{0.13,0.55,0.13}
\definecolor{GhostWhite}{rgb}{0.97,0.97,1.00}
\definecolor{GreenYellow}{rgb}{0.68,1.00,0.18}
\definecolor{HotPink1}{rgb}{1.00,0.43,0.71}
\definecolor{HotPink2}{rgb}{0.93,0.42,0.65}
\definecolor{HotPink3}{rgb}{0.80,0.38,0.56}
\definecolor{HotPink4}{rgb}{0.55,0.23,0.38}
\definecolor{HotPink}{rgb}{1.00,0.41,0.71}
\definecolor{IndianRed1}{rgb}{1.00,0.42,0.42}
\definecolor{IndianRed2}{rgb}{0.93,0.39,0.39}
\definecolor{IndianRed3}{rgb}{0.80,0.33,0.33}
\definecolor{IndianRed4}{rgb}{0.55,0.23,0.23}
\definecolor{IndianRed}{rgb}{0.80,0.36,0.36}
\definecolor{LavenderBlush1}{rgb}{1.00,0.94,0.96}
\definecolor{LavenderBlush2}{rgb}{0.93,0.88,0.90}
\definecolor{LavenderBlush3}{rgb}{0.80,0.76,0.77}
\definecolor{LavenderBlush4}{rgb}{0.55,0.51,0.53}
\definecolor{LavenderBlush}{rgb}{1.00,0.94,0.96}
\definecolor{LawnGreen}{rgb}{0.49,0.99,0.00}
\definecolor{LemonChiffon1}{rgb}{1.00,0.98,0.80}
\definecolor{LemonChiffon2}{rgb}{0.93,0.91,0.75}
\definecolor{LemonChiffon3}{rgb}{0.80,0.79,0.65}
\definecolor{LemonChiffon4}{rgb}{0.55,0.54,0.44}
\definecolor{LemonChiffon}{rgb}{1.00,0.98,0.80}
\definecolor{LightBlue1}{rgb}{0.75,0.94,1.00}
\definecolor{LightBlue2}{rgb}{0.70,0.87,0.93}
\definecolor{LightBlue3}{rgb}{0.60,0.75,0.80}
\definecolor{LightBlue4}{rgb}{0.41,0.51,0.55}
\definecolor{LightBlue}{rgb}{0.68,0.85,0.90}
\definecolor{LightCoral}{rgb}{0.94,0.50,0.50}
\definecolor{LightCyan1}{rgb}{0.88,1.00,1.00}
\definecolor{LightCyan2}{rgb}{0.82,0.93,0.93}
\definecolor{LightCyan3}{rgb}{0.71,0.80,0.80}
\definecolor{LightCyan4}{rgb}{0.48,0.55,0.55}
\definecolor{LightCyan}{rgb}{0.88,1.00,1.00}
\definecolor{LightGoldenrod1}{rgb}{1.00,0.93,0.55}
\definecolor{LightGoldenrod2}{rgb}{0.93,0.86,0.51}
\definecolor{LightGoldenrod3}{rgb}{0.80,0.75,0.44}
\definecolor{LightGoldenrod4}{rgb}{0.55,0.51,0.30}
\definecolor{LightGoldenrodYellow}{rgb}{0.98,0.98,0.82}
\definecolor{LightGoldenrod}{rgb}{0.93,0.87,0.51}
\definecolor{LightGray}{rgb}{0.83,0.83,0.83}
\definecolor{LightGreen}{rgb}{0.56,0.93,0.56}
\definecolor{LightGrey}{rgb}{0.83,0.83,0.83}
\definecolor{LightPink1}{rgb}{1.00,0.68,0.73}
\definecolor{LightPink2}{rgb}{0.93,0.64,0.68}
\definecolor{LightPink3}{rgb}{0.80,0.55,0.58}
\definecolor{LightPink4}{rgb}{0.55,0.37,0.40}
\definecolor{LightPink}{rgb}{1.00,0.71,0.76}
\definecolor{LightSalmon1}{rgb}{1.00,0.63,0.48}
\definecolor{LightSalmon2}{rgb}{0.93,0.58,0.45}
\definecolor{LightSalmon3}{rgb}{0.80,0.51,0.38}
\definecolor{LightSalmon4}{rgb}{0.55,0.34,0.26}
\definecolor{LightSalmon}{rgb}{1.00,0.63,0.48}
\definecolor{LightSeaGreen}{rgb}{0.13,0.70,0.67}
\definecolor{LightSkyBlue1}{rgb}{0.69,0.89,1.00}
\definecolor{LightSkyBlue2}{rgb}{0.64,0.83,0.93}
\definecolor{LightSkyBlue3}{rgb}{0.55,0.71,0.80}
\definecolor{LightSkyBlue4}{rgb}{0.38,0.48,0.55}
\definecolor{LightSkyBlue}{rgb}{0.53,0.81,0.98}
\definecolor{LightSlateBlue}{rgb}{0.52,0.44,1.00}
\definecolor{LightSlateGray}{rgb}{0.47,0.53,0.60}
\definecolor{LightSlateGrey}{rgb}{0.47,0.53,0.60}
\definecolor{LightSteelBlue1}{rgb}{0.79,0.88,1.00}
\definecolor{LightSteelBlue2}{rgb}{0.74,0.82,0.93}
\definecolor{LightSteelBlue3}{rgb}{0.64,0.71,0.80}
\definecolor{LightSteelBlue4}{rgb}{0.43,0.48,0.55}
\definecolor{LightSteelBlue}{rgb}{0.69,0.77,0.87}
\definecolor{LightYellow1}{rgb}{1.00,1.00,0.88}
\definecolor{LightYellow2}{rgb}{0.93,0.93,0.82}
\definecolor{LightYellow3}{rgb}{0.80,0.80,0.71}
\definecolor{LightYellow4}{rgb}{0.55,0.55,0.48}
\definecolor{LightYellow}{rgb}{1.00,1.00,0.88}
\definecolor{LimeGreen}{rgb}{0.20,0.80,0.20}
\definecolor{MediumAquamarine}{rgb}{0.40,0.80,0.67}
\definecolor{MediumBlue}{rgb}{0.00,0.00,0.80}
\definecolor{MediumOrchid1}{rgb}{0.88,0.40,1.00}
\definecolor{MediumOrchid2}{rgb}{0.82,0.37,0.93}
\definecolor{MediumOrchid3}{rgb}{0.71,0.32,0.80}
\definecolor{MediumOrchid4}{rgb}{0.48,0.22,0.55}
\definecolor{MediumOrchid}{rgb}{0.73,0.33,0.83}
\definecolor{MediumPurple1}{rgb}{0.67,0.51,1.00}
\definecolor{MediumPurple2}{rgb}{0.62,0.47,0.93}
\definecolor{MediumPurple3}{rgb}{0.54,0.41,0.80}
\definecolor{MediumPurple4}{rgb}{0.36,0.28,0.55}
\definecolor{MediumPurple}{rgb}{0.58,0.44,0.86}
\definecolor{MediumSeaGreen}{rgb}{0.24,0.70,0.44}
\definecolor{MediumSlateBlue}{rgb}{0.48,0.41,0.93}
\definecolor{MediumSpringGreen}{rgb}{0.00,0.98,0.60}
\definecolor{MediumTurquoise}{rgb}{0.28,0.82,0.80}
\definecolor{MediumVioletRed}{rgb}{0.78,0.08,0.52}
\definecolor{MidnightBlue}{rgb}{0.10,0.10,0.44}
\definecolor{MintCream}{rgb}{0.96,1.00,0.98}
\definecolor{MistyRose1}{rgb}{1.00,0.89,0.88}
\definecolor{MistyRose2}{rgb}{0.93,0.84,0.82}
\definecolor{MistyRose3}{rgb}{0.80,0.72,0.71}
\definecolor{MistyRose4}{rgb}{0.55,0.49,0.48}
\definecolor{MistyRose}{rgb}{1.00,0.89,0.88}
\definecolor{NavajoWhite1}{rgb}{1.00,0.87,0.68}
\definecolor{NavajoWhite2}{rgb}{0.93,0.81,0.63}
\definecolor{NavajoWhite3}{rgb}{0.80,0.70,0.55}
\definecolor{NavajoWhite4}{rgb}{0.55,0.47,0.37}
\definecolor{NavajoWhite}{rgb}{1.00,0.87,0.68}
\definecolor{NavyBlue}{rgb}{0.00,0.00,0.50}
\definecolor{OldLace}{rgb}{0.99,0.96,0.90}
\definecolor{OliveDrab1}{rgb}{0.75,1.00,0.24}
\definecolor{OliveDrab2}{rgb}{0.70,0.93,0.23}
\definecolor{OliveDrab3}{rgb}{0.60,0.80,0.20}
\definecolor{OliveDrab4}{rgb}{0.41,0.55,0.13}
\definecolor{OliveDrab}{rgb}{0.42,0.56,0.14}
\definecolor{OrangeRed1}{rgb}{1.00,0.27,0.00}
\definecolor{OrangeRed2}{rgb}{0.93,0.25,0.00}
\definecolor{OrangeRed3}{rgb}{0.80,0.22,0.00}
\definecolor{OrangeRed4}{rgb}{0.55,0.15,0.00}
\definecolor{OrangeRed}{rgb}{1.00,0.27,0.00}
\definecolor{PaleGoldenrod}{rgb}{0.93,0.91,0.67}
\definecolor{PaleGreen1}{rgb}{0.60,1.00,0.60}
\definecolor{PaleGreen2}{rgb}{0.56,0.93,0.56}
\definecolor{PaleGreen3}{rgb}{0.49,0.80,0.49}
\definecolor{PaleGreen4}{rgb}{0.33,0.55,0.33}
\definecolor{PaleGreen}{rgb}{0.60,0.98,0.60}
\definecolor{PaleTurquoise1}{rgb}{0.73,1.00,1.00}
\definecolor{PaleTurquoise2}{rgb}{0.68,0.93,0.93}
\definecolor{PaleTurquoise3}{rgb}{0.59,0.80,0.80}
\definecolor{PaleTurquoise4}{rgb}{0.40,0.55,0.55}
\definecolor{PaleTurquoise}{rgb}{0.69,0.93,0.93}
\definecolor{PaleVioletRed1}{rgb}{1.00,0.51,0.67}
\definecolor{PaleVioletRed2}{rgb}{0.93,0.47,0.62}
\definecolor{PaleVioletRed3}{rgb}{0.80,0.41,0.54}
\definecolor{PaleVioletRed4}{rgb}{0.55,0.28,0.36}
\definecolor{PaleVioletRed}{rgb}{0.86,0.44,0.58}
\definecolor{PapayaWhip}{rgb}{1.00,0.94,0.84}
\definecolor{PeachPuff1}{rgb}{1.00,0.85,0.73}
\definecolor{PeachPuff2}{rgb}{0.93,0.80,0.68}
\definecolor{PeachPuff3}{rgb}{0.80,0.69,0.58}
\definecolor{PeachPuff4}{rgb}{0.55,0.47,0.40}
\definecolor{PeachPuff}{rgb}{1.00,0.85,0.73}
\definecolor{PowderBlue}{rgb}{0.69,0.88,0.90}
\definecolor{RosyBrown1}{rgb}{1.00,0.76,0.76}
\definecolor{RosyBrown2}{rgb}{0.93,0.71,0.71}
\definecolor{RosyBrown3}{rgb}{0.80,0.61,0.61}
\definecolor{RosyBrown4}{rgb}{0.55,0.41,0.41}
\definecolor{RosyBrown}{rgb}{0.74,0.56,0.56}
\definecolor{RoyalBlue1}{rgb}{0.28,0.46,1.00}
\definecolor{RoyalBlue2}{rgb}{0.26,0.43,0.93}
\definecolor{RoyalBlue3}{rgb}{0.23,0.37,0.80}
\definecolor{RoyalBlue4}{rgb}{0.15,0.25,0.55}
\definecolor{RoyalBlue}{rgb}{0.25,0.41,0.88}
\definecolor{SaddleBrown}{rgb}{0.55,0.27,0.07}
\definecolor{SandyBrown}{rgb}{0.96,0.64,0.38}
\definecolor{SeaGreen1}{rgb}{0.33,1.00,0.62}
\definecolor{SeaGreen2}{rgb}{0.31,0.93,0.58}
\definecolor{SeaGreen3}{rgb}{0.26,0.80,0.50}
\definecolor{SeaGreen4}{rgb}{0.18,0.55,0.34}
\definecolor{SeaGreen}{rgb}{0.18,0.55,0.34}
\definecolor{SkyBlue1}{rgb}{0.53,0.81,1.00}
\definecolor{SkyBlue2}{rgb}{0.49,0.75,0.93}
\definecolor{SkyBlue3}{rgb}{0.42,0.65,0.80}
\definecolor{SkyBlue4}{rgb}{0.29,0.44,0.55}
\definecolor{SkyBlue}{rgb}{0.53,0.81,0.92}
\definecolor{SlateBlue1}{rgb}{0.51,0.44,1.00}
\definecolor{SlateBlue2}{rgb}{0.48,0.40,0.93}
\definecolor{SlateBlue3}{rgb}{0.41,0.35,0.80}
\definecolor{SlateBlue4}{rgb}{0.28,0.24,0.55}
\definecolor{SlateBlue}{rgb}{0.42,0.35,0.80}
\definecolor{SlateGray1}{rgb}{0.78,0.89,1.00}
\definecolor{SlateGray2}{rgb}{0.73,0.83,0.93}
\definecolor{SlateGray3}{rgb}{0.62,0.71,0.80}
\definecolor{SlateGray4}{rgb}{0.42,0.48,0.55}
\definecolor{SlateGray}{rgb}{0.44,0.50,0.56}
\definecolor{SlateGrey}{rgb}{0.44,0.50,0.56}
\definecolor{SpringGreen1}{rgb}{0.00,1.00,0.50}
\definecolor{SpringGreen2}{rgb}{0.00,0.93,0.46}
\definecolor{SpringGreen3}{rgb}{0.00,0.80,0.40}
\definecolor{SpringGreen4}{rgb}{0.00,0.55,0.27}
\definecolor{SpringGreen}{rgb}{0.00,1.00,0.50}
\definecolor{SteelBlue1}{rgb}{0.39,0.72,1.00}
\definecolor{SteelBlue2}{rgb}{0.36,0.67,0.93}
\definecolor{SteelBlue3}{rgb}{0.31,0.58,0.80}
\definecolor{SteelBlue4}{rgb}{0.21,0.39,0.55}
\definecolor{SteelBlue}{rgb}{0.27,0.51,0.71}
\definecolor{VioletRed1}{rgb}{1.00,0.24,0.59}
\definecolor{VioletRed2}{rgb}{0.93,0.23,0.55}
\definecolor{VioletRed3}{rgb}{0.80,0.20,0.47}
\definecolor{VioletRed4}{rgb}{0.55,0.13,0.32}
\definecolor{VioletRed}{rgb}{0.82,0.13,0.56}
\definecolor{WhiteSmoke}{rgb}{0.96,0.96,0.96}
\definecolor{YellowGreen}{rgb}{0.60,0.80,0.20}
\definecolor{aliceblue}{rgb}{0.94,0.97,1.00}
\definecolor{antiquewhite}{rgb}{0.98,0.92,0.84}
\definecolor{aquamarine1}{rgb}{0.50,1.00,0.83}
\definecolor{aquamarine2}{rgb}{0.46,0.93,0.78}
\definecolor{aquamarine3}{rgb}{0.40,0.80,0.67}
\definecolor{aquamarine4}{rgb}{0.27,0.55,0.45}
\definecolor{aquamarine}{rgb}{0.50,1.00,0.83}
\definecolor{azure1}{rgb}{0.94,1.00,1.00}
\definecolor{azure2}{rgb}{0.88,0.93,0.93}
\definecolor{azure3}{rgb}{0.76,0.80,0.80}
\definecolor{azure4}{rgb}{0.51,0.55,0.55}
\definecolor{azure}{rgb}{0.94,1.00,1.00}
\definecolor{beige}{rgb}{0.96,0.96,0.86}
\definecolor{bisque1}{rgb}{1.00,0.89,0.77}
\definecolor{bisque2}{rgb}{0.93,0.84,0.72}
\definecolor{bisque3}{rgb}{0.80,0.72,0.62}
\definecolor{bisque4}{rgb}{0.55,0.49,0.42}
\definecolor{bisque}{rgb}{1.00,0.89,0.77}
\definecolor{black}{rgb}{0.00,0.00,0.00}
\definecolor{blanchedalmond}{rgb}{1.00,0.92,0.80}
\definecolor{blue1}{rgb}{0.00,0.00,1.00}
\definecolor{blue2}{rgb}{0.00,0.00,0.93}
\definecolor{blue3}{rgb}{0.00,0.00,0.80}
\definecolor{blue4}{rgb}{0.00,0.00,0.55}
\definecolor{blueviolet}{rgb}{0.54,0.17,0.89}
\definecolor{blue}{rgb}{0.00,0.00,1.00}
\definecolor{brown1}{rgb}{1.00,0.25,0.25}
\definecolor{brown2}{rgb}{0.93,0.23,0.23}
\definecolor{brown3}{rgb}{0.80,0.20,0.20}
\definecolor{brown4}{rgb}{0.55,0.14,0.14}
\definecolor{brown}{rgb}{0.65,0.16,0.16}
\definecolor{burlywood1}{rgb}{1.00,0.83,0.61}
\definecolor{burlywood2}{rgb}{0.93,0.77,0.57}
\definecolor{burlywood3}{rgb}{0.80,0.67,0.49}
\definecolor{burlywood4}{rgb}{0.55,0.45,0.33}
\definecolor{burlywood}{rgb}{0.87,0.72,0.53}
\definecolor{cadetblue}{rgb}{0.37,0.62,0.63}
\definecolor{chartreuse1}{rgb}{0.50,1.00,0.00}
\definecolor{chartreuse2}{rgb}{0.46,0.93,0.00}
\definecolor{chartreuse3}{rgb}{0.40,0.80,0.00}
\definecolor{chartreuse4}{rgb}{0.27,0.55,0.00}
\definecolor{chartreuse}{rgb}{0.50,1.00,0.00}
\definecolor{chocolate1}{rgb}{1.00,0.50,0.14}
\definecolor{chocolate2}{rgb}{0.93,0.46,0.13}
\definecolor{chocolate3}{rgb}{0.80,0.40,0.11}
\definecolor{chocolate4}{rgb}{0.55,0.27,0.07}
\definecolor{chocolate}{rgb}{0.82,0.41,0.12}
\definecolor{coral1}{rgb}{1.00,0.45,0.34}
\definecolor{coral2}{rgb}{0.93,0.42,0.31}
\definecolor{coral3}{rgb}{0.80,0.36,0.27}
\definecolor{coral4}{rgb}{0.55,0.24,0.18}
\definecolor{coral}{rgb}{1.00,0.50,0.31}
\definecolor{cornflowerblue}{rgb}{0.39,0.58,0.93}
\definecolor{cornsilk1}{rgb}{1.00,0.97,0.86}
\definecolor{cornsilk2}{rgb}{0.93,0.91,0.80}
\definecolor{cornsilk3}{rgb}{0.80,0.78,0.69}
\definecolor{cornsilk4}{rgb}{0.55,0.53,0.47}
\definecolor{cornsilk}{rgb}{1.00,0.97,0.86}
\definecolor{cyan1}{rgb}{0.00,1.00,1.00}
\definecolor{cyan2}{rgb}{0.00,0.93,0.93}
\definecolor{cyan3}{rgb}{0.00,0.80,0.80}
\definecolor{cyan4}{rgb}{0.00,0.55,0.55}
\definecolor{cyan}{rgb}{0.00,1.00,1.00}
\definecolor{darkblue}{rgb}{0.00,0.00,0.55}
\definecolor{darkcyan}{rgb}{0.00,0.55,0.55}
\definecolor{darkgoldenrod}{rgb}{0.72,0.53,0.04}
\definecolor{darkgray}{rgb}{0.66,0.66,0.66}
\definecolor{darkgreen}{rgb}{0.00,0.39,0.00}
\definecolor{darkgrey}{rgb}{0.66,0.66,0.66}
\definecolor{darkkhaki}{rgb}{0.74,0.72,0.42}
\definecolor{darkmagenta}{rgb}{0.55,0.00,0.55}
\definecolor{darkolive}{rgb}{0.33,0.42,0.18}
\definecolor{darkorange}{rgb}{1.00,0.55,0.00}
\definecolor{darkorchid}{rgb}{0.60,0.20,0.80}
\definecolor{darkred}{rgb}{0.55,0.00,0.00}
\definecolor{darksalmon}{rgb}{0.91,0.59,0.48}
\definecolor{darksea}{rgb}{0.56,0.74,0.56}
\definecolor{darkslate}{rgb}{0.18,0.31,0.31}
\definecolor{darkslate}{rgb}{0.18,0.31,0.31}
\definecolor{darkslate}{rgb}{0.28,0.24,0.55}
\definecolor{darkturquoise}{rgb}{0.00,0.81,0.82}
\definecolor{darkviolet}{rgb}{0.58,0.00,0.83}
\definecolor{deeppink}{rgb}{1.00,0.08,0.58}
\definecolor{deepsky}{rgb}{0.00,0.75,1.00}
\definecolor{dimgray}{rgb}{0.41,0.41,0.41}
\definecolor{dimgrey}{rgb}{0.41,0.41,0.41}
\definecolor{dodgerblue}{rgb}{0.12,0.56,1.00}
\definecolor{firebrick1}{rgb}{1.00,0.19,0.19}
\definecolor{firebrick2}{rgb}{0.93,0.17,0.17}
\definecolor{firebrick3}{rgb}{0.80,0.15,0.15}
\definecolor{firebrick4}{rgb}{0.55,0.10,0.10}
\definecolor{firebrick}{rgb}{0.70,0.13,0.13}
\definecolor{floralwhite}{rgb}{1.00,0.98,0.94}
\definecolor{forestgreen}{rgb}{0.13,0.55,0.13}
\definecolor{gainsboro}{rgb}{0.86,0.86,0.86}
\definecolor{ghostwhite}{rgb}{0.97,0.97,1.00}
\definecolor{gold1}{rgb}{1.00,0.84,0.00}
\definecolor{gold2}{rgb}{0.93,0.79,0.00}
\definecolor{gold3}{rgb}{0.80,0.68,0.00}
\definecolor{gold4}{rgb}{0.55,0.46,0.00}
\definecolor{goldenrod1}{rgb}{1.00,0.76,0.15}
\definecolor{goldenrod2}{rgb}{0.93,0.71,0.13}
\definecolor{goldenrod3}{rgb}{0.80,0.61,0.11}
\definecolor{goldenrod4}{rgb}{0.55,0.41,0.08}
\definecolor{goldenrod}{rgb}{0.85,0.65,0.13}
\definecolor{gold}{rgb}{1.00,0.84,0.00}
\definecolor{gray0}{rgb}{0.00,0.00,0.00}
\definecolor{gray100}{rgb}{1.00,1.00,1.00}
\definecolor{gray10}{rgb}{0.10,0.10,0.10}
\definecolor{gray11}{rgb}{0.11,0.11,0.11}
\definecolor{gray12}{rgb}{0.12,0.12,0.12}
\definecolor{gray13}{rgb}{0.13,0.13,0.13}
\definecolor{gray14}{rgb}{0.14,0.14,0.14}
\definecolor{gray15}{rgb}{0.15,0.15,0.15}
\definecolor{gray16}{rgb}{0.16,0.16,0.16}
\definecolor{gray17}{rgb}{0.17,0.17,0.17}
\definecolor{gray18}{rgb}{0.18,0.18,0.18}
\definecolor{gray19}{rgb}{0.19,0.19,0.19}
\definecolor{gray1}{rgb}{0.01,0.01,0.01}
\definecolor{gray20}{rgb}{0.20,0.20,0.20}
\definecolor{gray21}{rgb}{0.21,0.21,0.21}
\definecolor{gray22}{rgb}{0.22,0.22,0.22}
\definecolor{gray23}{rgb}{0.23,0.23,0.23}
\definecolor{gray24}{rgb}{0.24,0.24,0.24}
\definecolor{gray25}{rgb}{0.25,0.25,0.25}
\definecolor{gray26}{rgb}{0.26,0.26,0.26}
\definecolor{gray27}{rgb}{0.27,0.27,0.27}
\definecolor{gray28}{rgb}{0.28,0.28,0.28}
\definecolor{gray29}{rgb}{0.29,0.29,0.29}
\definecolor{gray2}{rgb}{0.02,0.02,0.02}
\definecolor{gray30}{rgb}{0.30,0.30,0.30}
\definecolor{gray31}{rgb}{0.31,0.31,0.31}
\definecolor{gray32}{rgb}{0.32,0.32,0.32}
\definecolor{gray33}{rgb}{0.33,0.33,0.33}
\definecolor{gray34}{rgb}{0.34,0.34,0.34}
\definecolor{gray35}{rgb}{0.35,0.35,0.35}
\definecolor{gray36}{rgb}{0.36,0.36,0.36}
\definecolor{gray37}{rgb}{0.37,0.37,0.37}
\definecolor{gray38}{rgb}{0.38,0.38,0.38}
\definecolor{gray39}{rgb}{0.39,0.39,0.39}
\definecolor{gray3}{rgb}{0.03,0.03,0.03}
\definecolor{gray40}{rgb}{0.40,0.40,0.40}
\definecolor{gray41}{rgb}{0.41,0.41,0.41}
\definecolor{gray42}{rgb}{0.42,0.42,0.42}
\definecolor{gray43}{rgb}{0.43,0.43,0.43}
\definecolor{gray44}{rgb}{0.44,0.44,0.44}
\definecolor{gray45}{rgb}{0.45,0.45,0.45}
\definecolor{gray46}{rgb}{0.46,0.46,0.46}
\definecolor{gray47}{rgb}{0.47,0.47,0.47}
\definecolor{gray48}{rgb}{0.48,0.48,0.48}
\definecolor{gray49}{rgb}{0.49,0.49,0.49}
\definecolor{gray4}{rgb}{0.04,0.04,0.04}
\definecolor{gray50}{rgb}{0.50,0.50,0.50}
\definecolor{gray51}{rgb}{0.51,0.51,0.51}
\definecolor{gray52}{rgb}{0.52,0.52,0.52}
\definecolor{gray53}{rgb}{0.53,0.53,0.53}
\definecolor{gray54}{rgb}{0.54,0.54,0.54}
\definecolor{gray55}{rgb}{0.55,0.55,0.55}
\definecolor{gray56}{rgb}{0.56,0.56,0.56}
\definecolor{gray57}{rgb}{0.57,0.57,0.57}
\definecolor{gray58}{rgb}{0.58,0.58,0.58}
\definecolor{gray59}{rgb}{0.59,0.59,0.59}
\definecolor{gray5}{rgb}{0.05,0.05,0.05}
\definecolor{gray60}{rgb}{0.60,0.60,0.60}
\definecolor{gray61}{rgb}{0.61,0.61,0.61}
\definecolor{gray62}{rgb}{0.62,0.62,0.62}
\definecolor{gray63}{rgb}{0.63,0.63,0.63}
\definecolor{gray64}{rgb}{0.64,0.64,0.64}
\definecolor{gray65}{rgb}{0.65,0.65,0.65}
\definecolor{gray66}{rgb}{0.66,0.66,0.66}
\definecolor{gray67}{rgb}{0.67,0.67,0.67}
\definecolor{gray68}{rgb}{0.68,0.68,0.68}
\definecolor{gray69}{rgb}{0.69,0.69,0.69}
\definecolor{gray6}{rgb}{0.06,0.06,0.06}
\definecolor{gray70}{rgb}{0.70,0.70,0.70}
\definecolor{gray71}{rgb}{0.71,0.71,0.71}
\definecolor{gray72}{rgb}{0.72,0.72,0.72}
\definecolor{gray73}{rgb}{0.73,0.73,0.73}
\definecolor{gray74}{rgb}{0.74,0.74,0.74}
\definecolor{gray75}{rgb}{0.75,0.75,0.75}
\definecolor{gray76}{rgb}{0.76,0.76,0.76}
\definecolor{gray77}{rgb}{0.77,0.77,0.77}
\definecolor{gray78}{rgb}{0.78,0.78,0.78}
\definecolor{gray79}{rgb}{0.79,0.79,0.79}
\definecolor{gray7}{rgb}{0.07,0.07,0.07}
\definecolor{gray80}{rgb}{0.80,0.80,0.80}
\definecolor{gray81}{rgb}{0.81,0.81,0.81}
\definecolor{gray82}{rgb}{0.82,0.82,0.82}
\definecolor{gray83}{rgb}{0.83,0.83,0.83}
\definecolor{gray84}{rgb}{0.84,0.84,0.84}
\definecolor{gray85}{rgb}{0.85,0.85,0.85}
\definecolor{gray86}{rgb}{0.86,0.86,0.86}
\definecolor{gray87}{rgb}{0.87,0.87,0.87}
\definecolor{gray88}{rgb}{0.88,0.88,0.88}
\definecolor{gray89}{rgb}{0.89,0.89,0.89}
\definecolor{gray8}{rgb}{0.08,0.08,0.08}
\definecolor{gray90}{rgb}{0.90,0.90,0.90}
\definecolor{gray91}{rgb}{0.91,0.91,0.91}
\definecolor{gray92}{rgb}{0.92,0.92,0.92}
\definecolor{gray93}{rgb}{0.93,0.93,0.93}
\definecolor{gray94}{rgb}{0.94,0.94,0.94}
\definecolor{gray95}{rgb}{0.95,0.95,0.95}
\definecolor{gray96}{rgb}{0.96,0.96,0.96}
\definecolor{gray97}{rgb}{0.97,0.97,0.97}
\definecolor{gray98}{rgb}{0.98,0.98,0.98}
\definecolor{gray99}{rgb}{0.99,0.99,0.99}
\definecolor{gray9}{rgb}{0.09,0.09,0.09}
\definecolor{gray}{rgb}{0.75,0.75,0.75}
\definecolor{green1}{rgb}{0.00,1.00,0.00}
\definecolor{green2}{rgb}{0.00,0.93,0.00}
\definecolor{green3}{rgb}{0.00,0.80,0.00}
\definecolor{green4}{rgb}{0.00,0.55,0.00}
\definecolor{greenyellow}{rgb}{0.68,1.00,0.18}
\definecolor{green}{rgb}{0.00,1.00,0.00}
\definecolor{grey0}{rgb}{0.00,0.00,0.00}
\definecolor{grey100}{rgb}{1.00,1.00,1.00}
\definecolor{grey10}{rgb}{0.10,0.10,0.10}
\definecolor{grey11}{rgb}{0.11,0.11,0.11}
\definecolor{grey12}{rgb}{0.12,0.12,0.12}
\definecolor{grey13}{rgb}{0.13,0.13,0.13}
\definecolor{grey14}{rgb}{0.14,0.14,0.14}
\definecolor{grey15}{rgb}{0.15,0.15,0.15}
\definecolor{grey16}{rgb}{0.16,0.16,0.16}
\definecolor{grey17}{rgb}{0.17,0.17,0.17}
\definecolor{grey18}{rgb}{0.18,0.18,0.18}
\definecolor{grey19}{rgb}{0.19,0.19,0.19}
\definecolor{grey1}{rgb}{0.01,0.01,0.01}
\definecolor{grey20}{rgb}{0.20,0.20,0.20}
\definecolor{grey21}{rgb}{0.21,0.21,0.21}
\definecolor{grey22}{rgb}{0.22,0.22,0.22}
\definecolor{grey23}{rgb}{0.23,0.23,0.23}
\definecolor{grey24}{rgb}{0.24,0.24,0.24}
\definecolor{grey25}{rgb}{0.25,0.25,0.25}
\definecolor{grey26}{rgb}{0.26,0.26,0.26}
\definecolor{grey27}{rgb}{0.27,0.27,0.27}
\definecolor{grey28}{rgb}{0.28,0.28,0.28}
\definecolor{grey29}{rgb}{0.29,0.29,0.29}
\definecolor{grey2}{rgb}{0.02,0.02,0.02}
\definecolor{grey30}{rgb}{0.30,0.30,0.30}
\definecolor{grey31}{rgb}{0.31,0.31,0.31}
\definecolor{grey32}{rgb}{0.32,0.32,0.32}
\definecolor{grey33}{rgb}{0.33,0.33,0.33}
\definecolor{grey34}{rgb}{0.34,0.34,0.34}
\definecolor{grey35}{rgb}{0.35,0.35,0.35}
\definecolor{grey36}{rgb}{0.36,0.36,0.36}
\definecolor{grey37}{rgb}{0.37,0.37,0.37}
\definecolor{grey38}{rgb}{0.38,0.38,0.38}
\definecolor{grey39}{rgb}{0.39,0.39,0.39}
\definecolor{grey3}{rgb}{0.03,0.03,0.03}
\definecolor{grey40}{rgb}{0.40,0.40,0.40}
\definecolor{grey41}{rgb}{0.41,0.41,0.41}
\definecolor{grey42}{rgb}{0.42,0.42,0.42}
\definecolor{grey43}{rgb}{0.43,0.43,0.43}
\definecolor{grey44}{rgb}{0.44,0.44,0.44}
\definecolor{grey45}{rgb}{0.45,0.45,0.45}
\definecolor{grey46}{rgb}{0.46,0.46,0.46}
\definecolor{grey47}{rgb}{0.47,0.47,0.47}
\definecolor{grey48}{rgb}{0.48,0.48,0.48}
\definecolor{grey49}{rgb}{0.49,0.49,0.49}
\definecolor{grey4}{rgb}{0.04,0.04,0.04}
\definecolor{grey50}{rgb}{0.50,0.50,0.50}
\definecolor{grey51}{rgb}{0.51,0.51,0.51}
\definecolor{grey52}{rgb}{0.52,0.52,0.52}
\definecolor{grey53}{rgb}{0.53,0.53,0.53}
\definecolor{grey54}{rgb}{0.54,0.54,0.54}
\definecolor{grey55}{rgb}{0.55,0.55,0.55}
\definecolor{grey56}{rgb}{0.56,0.56,0.56}
\definecolor{grey57}{rgb}{0.57,0.57,0.57}
\definecolor{grey58}{rgb}{0.58,0.58,0.58}
\definecolor{grey59}{rgb}{0.59,0.59,0.59}
\definecolor{grey5}{rgb}{0.05,0.05,0.05}
\definecolor{grey60}{rgb}{0.60,0.60,0.60}
\definecolor{grey61}{rgb}{0.61,0.61,0.61}
\definecolor{grey62}{rgb}{0.62,0.62,0.62}
\definecolor{grey63}{rgb}{0.63,0.63,0.63}
\definecolor{grey64}{rgb}{0.64,0.64,0.64}
\definecolor{grey65}{rgb}{0.65,0.65,0.65}
\definecolor{grey66}{rgb}{0.66,0.66,0.66}
\definecolor{grey67}{rgb}{0.67,0.67,0.67}
\definecolor{grey68}{rgb}{0.68,0.68,0.68}
\definecolor{grey69}{rgb}{0.69,0.69,0.69}
\definecolor{grey6}{rgb}{0.06,0.06,0.06}
\definecolor{grey70}{rgb}{0.70,0.70,0.70}
\definecolor{grey71}{rgb}{0.71,0.71,0.71}
\definecolor{grey72}{rgb}{0.72,0.72,0.72}
\definecolor{grey73}{rgb}{0.73,0.73,0.73}
\definecolor{grey74}{rgb}{0.74,0.74,0.74}
\definecolor{grey75}{rgb}{0.75,0.75,0.75}
\definecolor{grey76}{rgb}{0.76,0.76,0.76}
\definecolor{grey77}{rgb}{0.77,0.77,0.77}
\definecolor{grey78}{rgb}{0.78,0.78,0.78}
\definecolor{grey79}{rgb}{0.79,0.79,0.79}
\definecolor{grey7}{rgb}{0.07,0.07,0.07}
\definecolor{grey80}{rgb}{0.80,0.80,0.80}
\definecolor{grey81}{rgb}{0.81,0.81,0.81}
\definecolor{grey82}{rgb}{0.82,0.82,0.82}
\definecolor{grey83}{rgb}{0.83,0.83,0.83}
\definecolor{grey84}{rgb}{0.84,0.84,0.84}
\definecolor{grey85}{rgb}{0.85,0.85,0.85}
\definecolor{grey86}{rgb}{0.86,0.86,0.86}
\definecolor{grey87}{rgb}{0.87,0.87,0.87}
\definecolor{grey88}{rgb}{0.88,0.88,0.88}
\definecolor{grey89}{rgb}{0.89,0.89,0.89}
\definecolor{grey8}{rgb}{0.08,0.08,0.08}
\definecolor{grey90}{rgb}{0.90,0.90,0.90}
\definecolor{grey91}{rgb}{0.91,0.91,0.91}
\definecolor{grey92}{rgb}{0.92,0.92,0.92}
\definecolor{grey93}{rgb}{0.93,0.93,0.93}
\definecolor{grey94}{rgb}{0.94,0.94,0.94}
\definecolor{grey95}{rgb}{0.95,0.95,0.95}
\definecolor{grey96}{rgb}{0.96,0.96,0.96}
\definecolor{grey97}{rgb}{0.97,0.97,0.97}
\definecolor{grey98}{rgb}{0.98,0.98,0.98}
\definecolor{grey99}{rgb}{0.99,0.99,0.99}
\definecolor{grey9}{rgb}{0.09,0.09,0.09}
\definecolor{grey}{rgb}{0.75,0.75,0.75}
\definecolor{honeydew1}{rgb}{0.94,1.00,0.94}
\definecolor{honeydew2}{rgb}{0.88,0.93,0.88}
\definecolor{honeydew3}{rgb}{0.76,0.80,0.76}
\definecolor{honeydew4}{rgb}{0.51,0.55,0.51}
\definecolor{honeydew}{rgb}{0.94,1.00,0.94}
\definecolor{hotpink}{rgb}{1.00,0.41,0.71}
\definecolor{indianred}{rgb}{0.80,0.36,0.36}
\definecolor{ivory1}{rgb}{1.00,1.00,0.94}
\definecolor{ivory2}{rgb}{0.93,0.93,0.88}
\definecolor{ivory3}{rgb}{0.80,0.80,0.76}
\definecolor{ivory4}{rgb}{0.55,0.55,0.51}
\definecolor{ivory}{rgb}{1.00,1.00,0.94}
\definecolor{khaki1}{rgb}{1.00,0.96,0.56}
\definecolor{khaki2}{rgb}{0.93,0.90,0.52}
\definecolor{khaki3}{rgb}{0.80,0.78,0.45}
\definecolor{khaki4}{rgb}{0.55,0.53,0.31}
\definecolor{khaki}{rgb}{0.94,0.90,0.55}
\definecolor{lavenderblush}{rgb}{1.00,0.94,0.96}
\definecolor{lavender}{rgb}{0.90,0.90,0.98}
\definecolor{lawngreen}{rgb}{0.49,0.99,0.00}
\definecolor{lemonchiffon}{rgb}{1.00,0.98,0.80}
\definecolor{lightblue}{rgb}{0.68,0.85,0.90}
\definecolor{lightcoral}{rgb}{0.94,0.50,0.50}
\definecolor{lightcyan}{rgb}{0.88,1.00,1.00}
\definecolor{lightgoldenrod}{rgb}{0.93,0.87,0.51}
\definecolor{lightgoldenrod}{rgb}{0.98,0.98,0.82}
\definecolor{lightgray}{rgb}{0.83,0.83,0.83}
\definecolor{lightgreen}{rgb}{0.56,0.93,0.56}
\definecolor{lightgrey}{rgb}{0.83,0.83,0.83}
\definecolor{lightpink}{rgb}{1.00,0.71,0.76}
\definecolor{lightsalmon}{rgb}{1.00,0.63,0.48}
\definecolor{lightsea}{rgb}{0.13,0.70,0.67}
\definecolor{lightsky}{rgb}{0.53,0.81,0.98}
\definecolor{lightslate}{rgb}{0.47,0.53,0.60}
\definecolor{lightslate}{rgb}{0.47,0.53,0.60}
\definecolor{lightslate}{rgb}{0.52,0.44,1.00}
\definecolor{lightsteel}{rgb}{0.69,0.77,0.87}
\definecolor{lightyellow}{rgb}{1.00,1.00,0.88}
\definecolor{limegreen}{rgb}{0.20,0.80,0.20}
\definecolor{linen}{rgb}{0.98,0.94,0.90}
\definecolor{magenta1}{rgb}{1.00,0.00,1.00}
\definecolor{magenta2}{rgb}{0.93,0.00,0.93}
\definecolor{magenta3}{rgb}{0.80,0.00,0.80}
\definecolor{magenta4}{rgb}{0.55,0.00,0.55}
\definecolor{magenta}{rgb}{1.00,0.00,1.00}
\definecolor{maroon1}{rgb}{1.00,0.20,0.70}
\definecolor{maroon2}{rgb}{0.93,0.19,0.65}
\definecolor{maroon3}{rgb}{0.80,0.16,0.56}
\definecolor{maroon4}{rgb}{0.55,0.11,0.38}
\definecolor{maroon}{rgb}{0.69,0.19,0.38}
\definecolor{mediumaquamarine}{rgb}{0.40,0.80,0.67}
\definecolor{mediumblue}{rgb}{0.00,0.00,0.80}
\definecolor{mediumorchid}{rgb}{0.73,0.33,0.83}
\definecolor{mediumpurple}{rgb}{0.58,0.44,0.86}
\definecolor{mediumsea}{rgb}{0.24,0.70,0.44}
\definecolor{mediumslate}{rgb}{0.48,0.41,0.93}
\definecolor{mediumspring}{rgb}{0.00,0.98,0.60}
\definecolor{mediumturquoise}{rgb}{0.28,0.82,0.80}
\definecolor{mediumviolet}{rgb}{0.78,0.08,0.52}
\definecolor{midnightblue}{rgb}{0.10,0.10,0.44}
\definecolor{mintcream}{rgb}{0.96,1.00,0.98}
\definecolor{mistyrose}{rgb}{1.00,0.89,0.88}
\definecolor{moccasin}{rgb}{1.00,0.89,0.71}
\definecolor{navajowhite}{rgb}{1.00,0.87,0.68}
\definecolor{navyblue}{rgb}{0.00,0.00,0.50}
\definecolor{navy}{rgb}{0.00,0.00,0.50}
\definecolor{oldlace}{rgb}{0.99,0.96,0.90}
\definecolor{olivedrab}{rgb}{0.42,0.56,0.14}
\definecolor{orange1}{rgb}{1.00,0.65,0.00}
\definecolor{orange2}{rgb}{0.93,0.60,0.00}
\definecolor{orange3}{rgb}{0.80,0.52,0.00}
\definecolor{orange4}{rgb}{0.55,0.35,0.00}
\definecolor{orangered}{rgb}{1.00,0.27,0.00}
\definecolor{orange}{rgb}{1.00,0.65,0.00}
\definecolor{orchid1}{rgb}{1.00,0.51,0.98}
\definecolor{orchid2}{rgb}{0.93,0.48,0.91}
\definecolor{orchid3}{rgb}{0.80,0.41,0.79}
\definecolor{orchid4}{rgb}{0.55,0.28,0.54}
\definecolor{orchid}{rgb}{0.85,0.44,0.84}
\definecolor{palegoldenrod}{rgb}{0.93,0.91,0.67}
\definecolor{palegreen}{rgb}{0.60,0.98,0.60}
\definecolor{paleturquoise}{rgb}{0.69,0.93,0.93}
\definecolor{paleviolet}{rgb}{0.86,0.44,0.58}
\definecolor{papayawhip}{rgb}{1.00,0.94,0.84}
\definecolor{peachpuff}{rgb}{1.00,0.85,0.73}
\definecolor{peru}{rgb}{0.80,0.52,0.25}
\definecolor{pink1}{rgb}{1.00,0.71,0.77}
\definecolor{pink2}{rgb}{0.93,0.66,0.72}
\definecolor{pink3}{rgb}{0.80,0.57,0.62}
\definecolor{pink4}{rgb}{0.55,0.39,0.42}
\definecolor{pink}{rgb}{1.00,0.75,0.80}
\definecolor{plum1}{rgb}{1.00,0.73,1.00}
\definecolor{plum2}{rgb}{0.93,0.68,0.93}
\definecolor{plum3}{rgb}{0.80,0.59,0.80}
\definecolor{plum4}{rgb}{0.55,0.40,0.55}
\definecolor{plum}{rgb}{0.87,0.63,0.87}
\definecolor{powderblue}{rgb}{0.69,0.88,0.90}
\definecolor{purple1}{rgb}{0.61,0.19,1.00}
\definecolor{purple2}{rgb}{0.57,0.17,0.93}
\definecolor{purple3}{rgb}{0.49,0.15,0.80}
\definecolor{purple4}{rgb}{0.33,0.10,0.55}
\definecolor{purple}{rgb}{0.63,0.13,0.94}
\definecolor{red1}{rgb}{1.00,0.00,0.00}
\definecolor{red2}{rgb}{0.93,0.00,0.00}
\definecolor{red3}{rgb}{0.80,0.00,0.00}
\definecolor{red4}{rgb}{0.55,0.00,0.00}
\definecolor{red}{rgb}{1.00,0.00,0.00}
\definecolor{rosybrown}{rgb}{0.74,0.56,0.56}
\definecolor{royalblue}{rgb}{0.25,0.41,0.88}
\definecolor{saddlebrown}{rgb}{0.55,0.27,0.07}
\definecolor{salmon1}{rgb}{1.00,0.55,0.41}
\definecolor{salmon2}{rgb}{0.93,0.51,0.38}
\definecolor{salmon3}{rgb}{0.80,0.44,0.33}
\definecolor{salmon4}{rgb}{0.55,0.30,0.22}
\definecolor{salmon}{rgb}{0.98,0.50,0.45}
\definecolor{sandybrown}{rgb}{0.96,0.64,0.38}
\definecolor{seagreen}{rgb}{0.18,0.55,0.34}
\definecolor{seashell1}{rgb}{1.00,0.96,0.93}
\definecolor{seashell2}{rgb}{0.93,0.90,0.87}
\definecolor{seashell3}{rgb}{0.80,0.77,0.75}
\definecolor{seashell4}{rgb}{0.55,0.53,0.51}
\definecolor{seashell}{rgb}{1.00,0.96,0.93}
\definecolor{sienna1}{rgb}{1.00,0.51,0.28}
\definecolor{sienna2}{rgb}{0.93,0.47,0.26}
\definecolor{sienna3}{rgb}{0.80,0.41,0.22}
\definecolor{sienna4}{rgb}{0.55,0.28,0.15}
\definecolor{sienna}{rgb}{0.63,0.32,0.18}
\definecolor{skyblue}{rgb}{0.53,0.81,0.92}
\definecolor{slateblue}{rgb}{0.42,0.35,0.80}
\definecolor{slategray}{rgb}{0.44,0.50,0.56}
\definecolor{slategrey}{rgb}{0.44,0.50,0.56}
\definecolor{snow1}{rgb}{1.00,0.98,0.98}
\definecolor{snow2}{rgb}{0.93,0.91,0.91}
\definecolor{snow3}{rgb}{0.80,0.79,0.79}
\definecolor{snow4}{rgb}{0.55,0.54,0.54}
\definecolor{snow}{rgb}{1.00,0.98,0.98}
\definecolor{springgreen}{rgb}{0.00,1.00,0.50}
\definecolor{steelblue}{rgb}{0.27,0.51,0.71}
\definecolor{tan1}{rgb}{1.00,0.65,0.31}
\definecolor{tan2}{rgb}{0.93,0.60,0.29}
\definecolor{tan3}{rgb}{0.80,0.52,0.25}
\definecolor{tan4}{rgb}{0.55,0.35,0.17}
\definecolor{tan}{rgb}{0.82,0.71,0.55}
\definecolor{thistle1}{rgb}{1.00,0.88,1.00}
\definecolor{thistle2}{rgb}{0.93,0.82,0.93}
\definecolor{thistle3}{rgb}{0.80,0.71,0.80}
\definecolor{thistle4}{rgb}{0.55,0.48,0.55}
\definecolor{thistle}{rgb}{0.85,0.75,0.85}
\definecolor{tomato1}{rgb}{1.00,0.39,0.28}
\definecolor{tomato2}{rgb}{0.93,0.36,0.26}
\definecolor{tomato3}{rgb}{0.80,0.31,0.22}
\definecolor{tomato4}{rgb}{0.55,0.21,0.15}
\definecolor{tomato}{rgb}{1.00,0.39,0.28}
\definecolor{turquoise1}{rgb}{0.00,0.96,1.00}
\definecolor{turquoise2}{rgb}{0.00,0.90,0.93}
\definecolor{turquoise3}{rgb}{0.00,0.77,0.80}
\definecolor{turquoise4}{rgb}{0.00,0.53,0.55}
\definecolor{turquoise}{rgb}{0.25,0.88,0.82}
\definecolor{violetred}{rgb}{0.82,0.13,0.56}
\definecolor{violet}{rgb}{0.93,0.51,0.93}
\definecolor{wheat1}{rgb}{1.00,0.91,0.73}
\definecolor{wheat2}{rgb}{0.93,0.85,0.68}
\definecolor{wheat3}{rgb}{0.80,0.73,0.59}
\definecolor{wheat4}{rgb}{0.55,0.49,0.40}
\definecolor{wheat}{rgb}{0.96,0.87,0.70}
\definecolor{whitesmoke}{rgb}{0.96,0.96,0.96}
\definecolor{white}{rgb}{1.00,1.00,1.00}
\definecolor{yellow1}{rgb}{1.00,1.00,0.00}
\definecolor{yellow2}{rgb}{0.93,0.93,0.00}
\definecolor{yellow3}{rgb}{0.80,0.80,0.00}
\definecolor{yellow4}{rgb}{0.55,0.55,0.00}
\definecolor{yellowgreen}{rgb}{0.60,0.80,0.20}
\definecolor{yellow}{rgb}{1.00,1.00,0.00}

\usepackage[usenames,dvipsnames]{xcolor}
\usepackage{amsmath,amsthm}
\usepackage[psamsfonts]{amsfonts}
\usepackage{graphicx}
\usepackage{natbib} 
\usepackage{amssymb}
\usepackage{enumitem}
\usepackage{float}
\usepackage{caption}
\usepackage{tabularx}
\usepackage{multicol}
\numberwithin{equation}{section}

\usepackage{hyperref}
\hypersetup{pdfpagemode=FullScreen,  
colorlinks=true,
citecolor=Bittersweet,
linkcolor=Bittersweet,
urlcolor=Bittersweet 
}

\newcommand{\R}{\ensuremath{\mathbb{R}}}

\newcommand{\T}{\ensuremath{{T}}}

\newtheorem{definition}{Definition}[section]
\newtheorem{theorem}{Theorem}[section]
\newtheorem{lemma}{Lemma}[section]

\begin{document}

\begin{center}
{\LARGE Tyler Shape Depth}	
\end{center}

\begin{center}
{\large Davy Paindaveine$^*$ \quad and \quad  Germain Van Bever$^\dagger$}	
\end{center}

\begin{center}
$^*$ ECARES and Departement of Mathematics, Universit\'{e} libre de Bruxelles, Avenue \mbox{F.D.} Roosevelt, 50, 
CP114/04, B-1050, Brussels, Belgium
\end{center}

\begin{center}
$^\dagger$ Departement of Mathematics and Namur Institute for Complex Systems, Universit\'{e} de Namur, Rempart de la Vierge, 8, 5000, Namur, Belgium
\end{center}

\begin{abstract}
In many problems from multivariate analysis, the parameter of interest is a shape matrix, that is, a normalized version  of the corresponding scatter or dispersion matrix. In this paper, we propose a depth concept for shape matrices that involves data points only through their directions from the center of the distribution. We use the terminology Tyler shape depth since the resulting estimator of shape, namely the deepest shape matrix, is the median-based counterpart of the M-estimator of shape of \cite{Tyl1987}. Beyond estimation, shape depth, like its Tyler antecedent, also allows hypothesis testing on shape. Its main benefit, however, lies in the ranking of shape matrices it provides, whose practical relevance is illustrated in principal component analysis and in shape-based outlier detection. We study the invariance, quasi-concavity and continuity properties of Tyler shape depth, the topological and boundedness properties of the corresponding depth regions, existence of a deepest shape matrix and prove Fisher consistency in the elliptical case. Finally, we derive a Glivenko--Cantelli-type result and establish almost sure consistency of the deepest shape matrix estimator.  
\end{abstract} 


\hrule


\section{Introduction}
\vspace{0mm}

Location depths measure the centrality of an arbitrary $k$-vector~$\theta$ with respect to a probabi\-lity measure~$P=P^X$ over~$\R^k$. Letting~$\mathcal{S}^{k-1}=\{x\in\R^k:\|x\|^2=x^\T x =1\}$ denote the unit sphere in~$\R^k$, the most famous instance is the   \cite{Tuk1975} halfspace depth
\begin{equation}
	\label{definHD}
	D(\theta,P)
=
\inf_{u\in \mathcal{S}^{k-1}}{\rm pr}\{u^\T(X-\theta)\geq 0\}
;
\end{equation}  
throughout, $\rm pr$ refers to probability under the probability measure~$P$ at hand.
 The halfspace depth regions~$\{\theta\in\R^k: D(\theta,P)\geq \alpha\}$ form a family of nested convex subsets of~$\R^k$. The Tukey median~$\theta_P$, defined as the barycenter of the innermost region~$M_P=\{\theta\in\R^k:D(\theta,P)=\max_{\xi\in\R^k} D(\xi,P)\}$, extends the univariate median to the multivariate case and is a robust alternative to the expectation~$E(X)$. Beyond location estimation, many inference problems can be tackled in a robust and nonparametric way by using the center-outward order resulting from depth (\citealp{Liuetal1999}). Adopting the parametric depth approach from \cite{Miz2002}, $D(\theta,P)$ can also be read as a measure of how well the location parameter value~$\theta$ fits the probability measure~$P$. In this spirit, possible outliers in a data set~$X_1,\ldots,X_n$ will be flagged by low depth values~$D(X_i,P_n)$, where~$P_n$ denotes the corresponding empirical probability measure. 

In this paper, the focus is on multivariate dispersion parameters known as shape matrices. For simplicity, we restrict in this section to elliptical distributions. Let~$\mathcal{P}_k$ be the collection of $k\times k$ symmetric positive definite matrices and write~$A^{1/2}$, with~$A\in\mathcal{P}_k$, for the unique square root of~$A$ in~$\mathcal{P}_k$. We will say that~$P=P^X$ is elliptical with location~$\theta\in\R^k$, scatter~$\Sigma\in\mathcal{P}_k$ and generating variate~$R$ if~$X$ has the same distribution as~$\theta+R\Sigma^{1/2}U$, where~$U$ is uniformly distributed over~$\mathcal{S}^{k-1}$ and is independent of the nonnegative scalar random variable~$R$, which has unit median. This median constraint makes~$\Sigma$ identifiable without moment conditions. Under finite second-order moments, the resulting covariance matrix is~$\Sigma_P=\{E(R^2)/k\}\Sigma$. Inference problems such as constructing confidence regions for~$\theta$ require one to estimate the full scatter matrix~$\Sigma$ or the full covariance matrix~$\Sigma_P$. However, in many other problems, it is sufficient to estimate the shape matrix, that is, the normalized scatter matrix
$$
V
=
\frac{k}{{\rm tr}(\Sigma)}\,\Sigma
=
\frac{k}{{\rm tr}(\Sigma_P)}\,\Sigma_P
.
$$ 
This shape matrix~$V$ could be normalized, as in \cite{Pai2008}, to have determinant one or upper-left entry one, which would not affect the results of the present paper. 
 For instance, principal components may be equivalently computed from~$V$, from~$\Sigma$ or, when it exists, from~$\Sigma_P$, since proportional matrices have the same eigenvectors. Now, when it comes to fixing the number of principal components on which to base further analysis, one typically looks at the proportions of explained variances 
$
p_m(\Sigma_P)
=
\sum_{\ell=1}^m \lambda_\ell(\Sigma_P)/\sum_{\ell=1}^k \lambda_\ell(\Sigma_P)
$ 
 ($m=1,\ldots,k$), where~$\lambda_\ell(A)$ denotes the $\ell$th largest eigenvalue of~$A$. Similarly to eigenvectors, these proportions remain unchanged if they are computed from~$V$ rather than from~$\Sigma$ or~$\Sigma_P$. In principal component analysis it is thus sufficient to estimate, or know the value of,~$V$.  

There is a large literature on inference for shape. Our main contribution is to provide a depth concept for shape, measuring how well a given shape matrix~$V$ fits the probability measure~$P$. While the proposed depth will lead to estimators and tests for shape, its main added value is the ordering of shape matrices resulting from depth. Here, we mention only two possible applications. The first is in principal component analysis, where a suitable estimator~$\hat{V}$ is to be chosen. When it is suspected that there might be outliers, one might for instance consider the minimum covariance determinant estimates~$\hat{V}_\gamma$, $\gamma\in[0.5,1]$, trimming a proportion~$1-\gamma$ of the data; see $\S$~\ref{Secapplis}. Choosing~$\gamma$ should typically be done on the basis of the proportion of outliers, which is usually unknown. We will show that the shape depth of~$\hat{V}_{\gamma}$ allows for an informed choice on~$\gamma$. The second application concerns outlier detection in multivariate financial times series. Since volatility is key in finance, one might flag atypical days in such series by spotting days that associate a low depth to a shape estimator~$\hat V_{\rm full}$ computed from the full series.

Depth for a generic parameter has been discussed in~\cite{Miz2002}. Depth for scatter matrices, however, has only been considered in \cite{Zha2002}, \cite{Chenetal16} and \cite{PVB17}, and only the last considers depth for shape matrices.


\section{Shape depth}
\vspace{0mm}

\cite{Tyl1987} introduced a shape notion extending the concept of shape outside the elliptical setup. Consider the multivariate sign~$U_{\theta,V}$ defined as~$V^{-1/2}(X-\theta)/\|V^{-1/2}(X-\theta)\|$ if~$X\neq \theta$ and as~$0$ otherwise, where~$V^{-1/2}$ is the inverse of~$V^{1/2}$.
%
Let also~$W_{\theta,V}={\rm vec}\{ U_{\theta,V} U^\T_{\theta,V} - (1/k) I_k\}$, where ${\rm vec}\,A$ stacks the columns of~$A$ on top of each other and where~$I_k$ is the $k\times k$ identity matrix. The Tyler shape of~$P=P^X$, $V_T$ say, is then the matrix~$V\in\mathcal{P}_{k,{\rm tr}}=\{V\in\mathcal{P}_k:~{\rm tr}(V)=k \}$ satisfying
\begin{equation}
\label{tyler}	
E 
(
W_{\theta,V}
)
=
0
.
\end{equation} 
If~$P$ is smooth at~$\theta$, in the sense that no hyperplane containing $\theta$ has a strictly positive \mbox{$P$-probability} mass, then~(\ref{tyler}) admits a unique solution~$V\in\mathcal{P}_{k,{\rm tr}}$ that agrees with the true shape if~$P$ is elliptical with location~$\theta$ (\citealp{Tyl1987}; \citealp{KenTyl1988}; \citealp{Dum1998}).
In essence, (\ref{tyler}) identifies the shape~$V$ making the origin of~$\R^{k^2}$ most central in an~$L_2$-sense for the distribution~$P^{W_{\theta,V}}$ of~$W_{\theta,V}$, that is, it defines $V_T$ as the solution of
\begin{equation}
\label{heodsh}	
0
=
\arg\min_{m\in\R^{k^2}}
 E
(
\left\|W_{\theta,V}-m\right\|^2
).
\end{equation}

The present work finds its source in the idea that one may define the shape of~$P$ as the matrix~$V\in\mathcal{P}_{k,{\rm tr}}$ making the origin of~$\R^{k^2}$ most central for the distribution of $W_{\theta, V}$, in the halfspace depth sense, that is, as the value of~$V$ maximizing the following depth.

\begin{definition}[Tyler shape depth]
\label{shapedepthD}
Let~$P=P^X$ be a probability measure over~$\R^k$ and fix~$V\in\mathcal{P}_{k,{\rm tr}}$. (i) For any~$\theta\in\R^k$, the fixed-$\theta$ shape depth of~$V$ with respect to~$P$ 
is
$
D_\theta(V,P)
=
D(0,P^{W_{\theta,V}}) 
=
 \inf_{u\in\mathcal{S}^{k^2-1}}
 {\rm pr}(
  u^\T 
W_{\theta,V} 
 \geq 0 
)
$.
(ii) 
The shape depth of~$\hspace{.2mm}V$ with respect to~$P$ 
is~$D(V,P)=D_{\theta_P}(V,P)$, where~$\theta_P$ is the Tukey median of~$P$. 
\end{definition}

We will use the notation~$D(\cdot,P)$ for both halfspace and Tyler shape depths, as the vector or matrix nature of the argument will remove any ambiguity. The fixed-$\theta$ shape depth can equivalently be defined 	as
$
D_\theta(V,P)
=
 \inf_{M}
 {\rm pr}\{
U_{\theta,V}^\T M U_{\theta,V} 
 -  {{\rm tr}(M)/k}  
 \geq 0  
\}
,
$
where the infimum is over all $k\times k$ symmetric matrices~$M$; see Lemma~1 in the Supplementary Material.
While, in view of~(\ref{heodsh}),~$V_T$ can be seen as a sign-based mean concept for shape, the maximizer of Tyler shape depth is of a median nature. The main benefit of the proposed depth does not come from the deepest shape itself but rather from the ranking of shapes it provides; see $\S$~\ref{Secapplis}. 

Definition~\ref{shapedepthD}(ii) calls for some comments. 
Two approaches were considered in the literature for Tyler shape in the case of unspecified center: the \cite{Tyl1987} plug-in approach, which replaces the unknown~$\theta$ with some location functional, and the \cite{HetRan2002} approach, which jointly solves
$
E
(
U_{\theta,V} 
)
=
0
$ and $
E
(
W_{\theta,V}
) 
=
0
$; 
existence of a unique solution to joint location and scatter M-estimating equations was studied in \cite{Mar1976} under ellipticity and in \cite{TatTyl2000} for non-elliptic cases. Both approaches provide two distinct shapes outside the elliptical setup. In contrast, for the proposed depth, the plug-in and joint maximization approaches always lead to the same shape: irrespective of $\lambda$, the objective function $(\theta,V)\mapsto D(0,P^{U_{\theta,V}})+\lambda D(0,P^{W_{\theta,V}})$ is indeed maximized 
\vspace{-.8mm}
at $\theta=\theta_P$ and $V=\arg\max_V D(0,P^{W_{\theta_P,V}})$, since $D\left(0,P^{U_{\theta,V}}\right)=D(0,P^{V^{-1/2}(X-\theta)})=D(\theta,P^X)$ is, for any $V$, maximized at~$\theta=\theta_P$.

An alternative way to obtain an unspecified location version of Tyler shape is to construct it on pairwise differences (\citealp{Dum1998}). We will not investigate this for our shape depth, since the sample version of the resulting depth would lead to a much heavier computational burden.


\section{Main properties}
\vspace{0mm}

In this section, we study the main properties of the shape depth~$D_\theta(V,P)$ and of the corresponding depth regions~$R_{\theta}(\alpha,P)=\{ V\in \mathcal{P}_{k,{\rm tr}}: D_\theta(V,P)\geq \alpha\}$. Topological statements for subsets of~$\mathcal{P}_{k,{\rm tr}}$ and for functions defined on~$\mathcal{P}_{k,{\rm tr}}$ will refer to the topology whose open sets are generated by balls of the form $B(V_0,r)=\{ V\in\mathcal{P}_{k,{\rm tr}} : d(V,V_0)<r\}$, where~$d$ is the usual geodesic distance   on~$\mathcal{P}_k$: with the classical log mapping on~$\mathcal{P}_k$, this distance is such that~$
d(V_a,V_b)=\| \log( V_a^{-1/2} V_b V_a^{-1/2} ) \|_F 
$, where~$\|A\|_F=\{{\rm tr}(AA^\T)\}^{1/2}$ is the Frobenius norm of~$A$ (\citealp{Bha07}). We start with the following continuity result.

\begin{theorem}
\label{Theorcontinuity}
Let~$P$ be a probability measure over~$\R^k$ and fix~$\theta\in\R^k$.
Then, (i)  $V\mapsto D_\theta(V,P)$ is upper semicontinuous on~$\mathcal{P}_{k,{\rm tr}}$; (ii) the depth region~$R_{\theta}(\alpha,P)$ is closed for any~\mbox{$\alpha\geq 0$}; (iii) if $P$ is absolutely continuous with respect to the Lebesgue measure, then $V\mapsto D_\theta(V,P)$ is also lower semicontinuous, hence continuous, on~$\mathcal{P}_{k,{\rm tr}}$. 
\end{theorem}

We will say that a subset~$R$ of~$\mathcal{P}_{k,{\rm tr}}$ is bounded if and only if~$R\subset B(I_k,r)$ for some~$r>0$; since~$d$ satisfies the triangle inequality, we need only consider balls centered at~$I_k$. 
Moreover, we will say that $P$ is smooth at~$\theta$ if and only if~$t_{\theta,P}=0$, with~$t_{\theta,P}=\sup_{u\in\mathcal{S}^{k-1}} {\rm pr}\{u^\T(X-\theta)=0\}$. We then have the following result. 

\begin{theorem}
\label{Theorboundedness}
Let~$P$ be a probability measure over~$\R^k$ and fix~$\theta\in\R^k$. Then the depth region~$R_{\theta}(\alpha,P)$ is bounded and compact for any~$\alpha>t_{\theta,P}$.
\end{theorem}

The main reason to work with geodesic distance rather than Frobenius distance $d_F(V_1,V_2)=\|V_2-V_1\|_F$ is that, unlike~$(\mathcal{P}_{k,{\rm tr}},d_F)$, the metric space~$(\mathcal{P}_{k,{\rm tr}},d)$ is complete; see, e.g., Proposition~10 in \cite{BhaHol06}. This is what allows us to establish compacity in Theorem~\ref{Theorboundedness}, which is the main ingredient for the following result.

\begin{theorem}
\label{Theorexistmax}
Let~$P$ be a probability measure over~$\R^k$ and fix~$\theta\in\R^k$. (i) If~$R_{\theta}(t_{\theta,P},P)$ is non-empty, then there exists a shape~$V_*\in\mathcal{P}_{k,{\rm tr}}$ maximizing~$D_\theta(V,P)$. In particular, (ii) if~$P$ is smooth at~$\theta$, then such a deepest shape~$V_*$ exists.    
\end{theorem}

While the previous result guarantees existence of a deepest shape for absolutely continuous probability measures, uniqueness is not guaranteed in general. Parallel to what is done for the Tukey median, we then define the fixed-$\theta$ shape matrix of~$P$ as the barycenter of the deepest shape region of~$P$, that is, as the shape matrix~$V_{\theta,P}$ satisfying
\begin{equation}
\label{barydef}	
{\rm vec}\,V_{\theta,P}
=
{\int_{{\rm vec}\,R_\theta(\alpha_*,P)} v \, dv}
\Big/
{\int_{{\rm vec}\,R_\theta(\alpha_*,P)} dv}
,
\end{equation}
with~$\alpha_*=\max_V D_\theta(V,P)$. Two remarks are in order. First, the integrals in~(\ref{barydef}) exist and are finite since~${\rm vec}\,\mathcal{P}_{k,{\rm tr}}$ is a bounded subset of~$\R^{k^2}$: $0\leq V^2_{ij} < V_{ii} V_{jj}\leq k^2$ for any~$V
\in\mathcal{P}_{k,{\rm tr}}$. Second,~the following convexity result implies that $V_{\theta,P}$ has maximal depth.

\begin{theorem}
	\label{TheorQuasic}
Let~$P$ be a probability measure over~$\R^k$ and fix~$\theta\in\R^k$.
Then, (i)  $V\mapsto D_\theta(V,P)$ is quasi-concave: $D_\theta(V_t,P) \geq \min\{D_\theta(V_a,P),D_\theta(V_b,P)\}$ for~$V_t=(1-t)V_a+t V_b$ with~$V_a,V_b\in \mathcal{P}_{k,{\rm tr}}$ and~$t\in[0,1]$; (ii) the region~$R_{\theta}(\alpha,P)$ is convex  for any~$\alpha\geq 0$.
\end{theorem}

This defines the fixed-$\theta$ shape of a probability measure~$P$ under the very mild condition that~$R_{\theta}(t_{\theta,P},P)$ is non-empty, hence in particular when~$P$ is smooth at~$\theta$. Of course, it is important that, under ellipticity, this agrees with the elliptical concept of shape provided in $\S$\,1. The following Fisher consistency result confirms that this is the case.

\begin{theorem}
\label{TheorFishconst}
Let~$P$ be an elliptical probability measure over~$\R^k$ with location~$\theta_0$ and shape~$V_0$. Then,
$
D_{\theta_0}(V_0,P)
 \geq 
D_{\theta_0}(V,P)
$
for any~$V\in\mathcal{P}_{k,{\rm tr}}$, and, provided \mbox{that~${\rm pr}[\{\theta_0\}]<1$,} the equality holds if and only if 
$V=V_0$. Letting~$Y_k$ be Beta with parameters~$1/2$ and~$(k-1)/2$, the maximal depth is~$D_{\theta_0}(V_0,P)=(1-{\rm pr}[\{\theta_0\}]) {\rm pr}(Y_k>1/k)$. 
\end{theorem} 


In this result, ${\rm pr}[\{\theta_0\}]$ equals the probability that the generating variate~$R$ associated to~$P$ is equal to zero. Lemma~2 in \cite{PVB17b} implies that the maximal depth in Theorem~\ref{TheorFishconst} is monotone decreasing in~$k$ if~${\rm pr}[\{\theta_0\}]$ does not depend on~$k$, in which case the maximal depth is convergent as~$k$ goes to infinity. Since~$Y_k$ has the same distribution as~$Z_1^2/(\sum_{\ell=1}^k Z_\ell^2)$, where~$Z=(Z_1,\ldots,Z_k)^\T$ is $k$-variate standard normal, the limit is equal to~${\rm pr}(Z_1^2>1)\approx 0.317$.  
The proof of Theorem~\ref{TheorFishconst} requires the following result.

\begin{theorem}
\label{Theoraffineinvariance}
Let~$P=P^X$ be a probability measure over~$\R^k$ and fix~$\theta\in\R^k$.
Then, for any shape matrix~$V$, any invertible $k\times k$ matrix~$A$ and any $k$-vector~$b$, 
$$
D_{A\theta+b}\big(V_A,P^{A X+{b}}\big)
=
D_\theta(V,P^X)
, \quad 
R_{A\theta+b}(\alpha,P^{AX+b})
=
\big\{
V_A : V\in R_{\theta}(\alpha,P)
\big\}
,
$$
where~$V_A=k A V\! A^\T / {\rm tr}(A V\!A^\T)$ is the shape matrix proportional to~$A V\!A^\T$.
\end{theorem}

This shows that the fixed-$\theta$ shape depth and the corresponding regions behave well under affine transformations, and in particular under changes of the measurement units. Affine invariance is a classical requirement in location depth (\citealp{ZuoSer2000A}). 

Tyler shape depth is a sign concept in the sense that it depends on the underlying random vector~$X$ only through its multivariate sign~$U_{\theta,V}$. In the elliptical case, it follows that, if the distribution does not charge the center of the distribution, this depth does not depend on the distribution of the underlying generating variate~$R$. More precisely, we have the following result. 

\begin{theorem}
\label{TheorDistribfreeness}
Let~$P$ be an elliptical probability measure over~$\R^k$ with location~$\theta_0$ and shape~$V_0$. Then, (i) for some~$h:\mathcal{P}_{k,{\rm tr}}\to [0,1]$ that does not depend on~$V$ or on~$P$,
\begin{equation}
	\label{theconj}
D_{\theta_0}(V,P)
=
(1-{\rm pr}[\{\theta_0\}])
\,
h\bigg\{ \frac{k (V_0^{-1/2} V V_0^{-1/2})}{{\rm tr}(V_0^{-1}V)}\bigg\}
;
\end{equation}
(ii) for~$k=2$, 
\begin{equation}
	\label{explicik2}
D_{\theta_0}(V,P)
=
(1-{\rm pr}[\{\theta_0\}])
\,
{\rm pr}\bigg(
Y_2
\geq
\frac{1}{2}
+
\frac{1}{2} 
\bigg[
1- \det \bigg\{ \frac{2 V_0^{-1} V}{{\rm tr}(V_0^{-1}V)}\bigg\}
\bigg]^{1/2} 
\,
\bigg)
,
\end{equation}
with~$Y_2$ is Beta distributed with parameters~$1/2$ and~$1/2$.
\end{theorem} 

The function~$h$ in this result does not depend on~$P$, so that depth, under ellipticity, depends on~$P$ through~$V_0$ and~${\rm pr}[\{\theta_0\}]$ only, with the dependence on~${\rm pr}[\{\theta_0\}]$ not affecting the induced ranking of shape matrices. 
%
It is easy to check that the explicit bivariate elliptical depth in~(\ref{explicik2}) is compatible with the general results obtained above. While it seems very challenging to obtain an explicit expression for the function~$h$ in~(\ref{theconj}), numerical experiments lead us to conjecture that, irrespective of the dimension~$k$, the mapping~$h$ is of the form~$h(M)=g(\det M)$ for some function~$g:\R^+\to[0,1]$.  

The results of this section extend to the unspecified-location shape depth~$D(V,P)=D_{\theta_P}(V,P)$ and to the corresponding regions~$R(\alpha,P)=\{ V\in \mathcal{P}_{k,{\rm tr}}: D(V,P)\geq \alpha\}$. Theorems~\ref{Theorcontinuity} to~\ref{TheorQuasic} hold for any fixed~$\theta$ and their unspecified-$\theta$ versions are simply obtained by substituting~$\theta_P$ for~$\theta$ throughout. In particular, the existence of an unspecified-location deepest shape matrix is guaranteed if~$P$ is smooth at~$\theta_P$, or, more generally, if~$R(t_{\theta_P,P},P)$ is non-empty. Under unspecified location, the shape~$V_P$ of~$P$ is then defined as the barycenter of the set of shape matrices maximizing~$D(\cdot,P)$.  In view of the affine equivariance of~$\theta_P$, i.e.,~$\theta_{P^{AX+B}}=A\theta_{P^X}+b$, the affine-invariance/equivariance properties
$$
D\big(V_A,P^{A X+{b}}\big)
=
D(V,P^X)
, \quad 
R(\alpha,P^{AX+b})
=
\big\{
V_A : V\in R(\alpha,P)
\big\}
$$
follow directly from Theorem~\ref{Theoraffineinvariance}, to which we refer for the definition of~$V_A$. Finally, Theorems~\ref{TheorFishconst} and~\ref{TheorDistribfreeness} also readily extend to the unspecified-location case, since~$\theta_P=\theta_0$ for any elliptical probability measure~$P$ with location~$\theta_0$. In particular, if~$P$ is elliptical with shape~$V_0$, then the unspecified-$\theta$ shape depth~$D(V,P)$ is uniquely maximized at~$V=V_0$, if the distribution is not degenerate at a single point.

\section{Consistency}
\vspace{0mm}
\label{SecConsist}

When $k$-variate observations~$X_1,\ldots,X_n$ are available, we define the sample fixed-$\theta$ depth of a shape matrix~$V$ as~$D_\theta(V,P_n)$, where~$P_n$ is the empirical probability measure associated with~$X_1,\ldots,X_n$, and its unspecified-location version as~$D(V,P_n)$. In this section, we state a Glivenko--Cantelli-type result for these sample depths and investigate consistency of max-depth shape estimators.

\begin{theorem}
\label{Theorunifconst}
Let~$P$ be a probability measure over~$\R^k$ and let~$P_n$ denote the empirical probability measure associated with a random sample of size~$n$ from~$P$. Then, (i) for any~$\theta\in\R^k$, $\sup_{V\in\mathcal{P}_{k,{\rm tr}}}
|D_\theta(V,P_n)-D_\theta(V,P)| \to 0$ almost surely as~$n\to\infty$; (ii) if~$P$ is absolutely continuous with respect to the Lebesgue measure, then  $\sup_{V\in\mathcal{P}_{k,{\rm tr}}} |D(V,P_n)-D(V,P)| \to 0$ almost surely as~$n\to\infty$. 
\end{theorem} 

We illustrate 
this result in the bivariate elliptical case associated with Theorem~\ref{TheorDistribfreeness}(ii). Figure~\ref{Fig1} provides contour plots of~$D_\theta(V,P)$ in terms of~$V_{12}/(V_{11}V_{22})^{1/2}$ and~$V_{22}/V_{11}$, 
for various bivariate, arbitrarily elliptical, probability measures. The sign nature of shape depth ensures that these contours, along with their empirical counterparts, are distribution-free in the class of elliptical distributions that do not charge the centre of symmetry.
Figure~\ref{Fig1} also reports the empirical contour plots obtained from a random sample of size~$n=800$ drawn from the corresponding bivariate normal distributions. Clearly, the results support the consistency in Theorem~\ref{Theorunifconst}(i).

In $\S$~3, the shape~$V_{\theta,P}$ of~$P$ was defined as the barycenter of the collection of $P$-deepest shape matrices. In the empirical case, a natural estimator is the corresponding shape matrix~$V_{\theta,P_n}$ computed from the empirical probability measure~$P_n$ associated with the sample at hand;
existence here follows from the fact that~$D_\theta(V,P_n)$ may only take values~$\ell/n$ ($\ell=0,1,\ldots,n$). The same argument ensures the existence of the sample deepest shape~$V_{P_n}$ in the unspecified-location case. The sample Tukey median~$\theta_{P_n}$ was one of the first affine-equivariant location estimators with a high breakdown point. It would therefore be interesting to investigate whether the affine-equivariant shape estimator~$V_{P_n}$, parallel to the Maronna--Stahel--Yohai P-estimators of scatter, also has a high breakdown point (\citealp{Tyl1994}). Since this is beyond the scope of this paper, we focus on consistency of sample deepest shapes.  

\begin{theorem}
\label{Theormaxdepthconst}
Let~$P$ be a probability measure over~$\R^k$ and let~$P_n$ denote the empirical probability measure associated with a random sample of size~$n$ from~$P$. (i) Fix~$\theta\in\R^k$ and assume that~$R_{\theta}(t_{\theta,P},P)$ is non-empty. Then, $V_{\theta,P_n}\to V_{\theta,P}$ almost surely as~$n\to\infty$. 
(ii) If~$P$ is absolutely continuous with respect to the Lebesgue measure, then~$V_{P_n}\to V_{P}$ almost surely as~$n\to\infty$. 
\end{theorem}

The specified-$\theta$ result in Theorem~\ref{Theormaxdepthconst}(i) holds in particular 
if~$P$ is smooth at~$\theta$. The unspecified-$\theta$ result requires a more stringent smoothness assumption, namely absolute continuity of~$P$. This assumption, which is already present in Theorem~\ref{Theorunifconst}(ii), is only needed to control the impact of replacing~$\theta$ by~$\theta_{P_n}$ in~$D_\theta(V,P_n)$ and~$V_{\theta,P_n}$. Figure~\ref{Fig1} also supports Theorem~\ref{Theormaxdepthconst}(i) since, in each sample considered, the sample deepest shape is close to its population counterpart.
%
%
%

\section{Two applications}
\vspace{0mm}
\label{Secapplis}

\subsection{Choosing a shape matrix estimator in principal component analysis}
\vspace{0mm}
\label{SecAppliPCA}

There is a vast literature on scatter or shape estimation. Among the most famous estimators are 
the minimum covariance determinant scatters~$S_\gamma$. 
Recall that, in the empirical case, $S_\gamma$ is the covariance matrix with the smallest determinant among covariance matrices computed using only a proportion $\gamma$ of the observations. The choice of the trimming proportion $1-\gamma$ is crucial, as the loss in efficiency can be very large if the trimming is excessive; see, for example, \cite{CroHae1999} or \cite{PaiVanB2014}. Choosing $\gamma$ is therefore difficult, as it should be taken large, but not so large as to incorporate outliers. In this section, we consider robust principal component analysis based on the shape estimators~$\hat{V}_\gamma=kS_\gamma/{\rm tr}(S_\gamma)$ and show that Tyler shape depth allows the making of an informed choice of~$\gamma$.

For several contamination proportions~$\eta$, we independently generated~$R=500$ bivariate samples of~$n=800$ independent observations, each comprising~$(1-\eta)n$ clean observations and~$\eta n$ outliers. With~$X$ bivariate normal with zero mean and covariance matrix~${\rm diag}(4,1)$
 and~$Y$ bivariate normal with mean~$(0,\delta)^\T$ and identity covariance matrix, the clean observations are equal to~$X$ in distribution, whereas the outliers are distributed, in equal proportions, as~$Y$ or~$-Y$. Two simulations were conducted, one for~$\delta=4$ and one for~$\delta=5$; clearly, the former simulation provides a harder robustness problem than the latter. We consider estimating the first principal direction~$e_1=(1,0)^\T$ of the uncontaminated distribution. For any~$\gamma\in[0.5,1]$, a natural estimator is, up to a sign, the first eigenvector~$\hat{v}_{\gamma}$ of~$\hat{V}_\gamma$. Denoting as~$\hat{v}_{r,\gamma}$ this estimate in replication~$r=1,\ldots,R$, estimation performance can be measured through the mean squared error 
$$
\textrm{MSE}_\gamma
=
\frac{1}{R}
\sum_{r=1}^R
\,
(\Delta\alpha_{r,\gamma})^2
,
$$
where~$\Delta\alpha_{r,\gamma}=\arccos ( |e_1^\T \hat{v}_{r,\gamma}| )$ is the angle between the population first eigendirection~$e_1$ and its estimate~$\hat{v}_{r,\gamma}$. Figure~\ref{Fig2} plots~$\textrm{MSE}_\gamma$ as a function of~$\gamma$; the Monte Carlo exercise was performed for every value of~$\gamma\in\{0.5,0.51,\ldots,0.99,1\}$. The results confirm that, for any contamination proportion~$\eta$, a suitable value of~$\gamma$ should be identified. The optimal value~$\gamma_0=\arg\min_\gamma \textrm{MSE}_\gamma$ basically coincides with~$1-\eta$ in the easy case~$\delta=5$, whereas, in the harder one~$\delta=4$, $\gamma_0$ is slightly smaller than~$1-\eta$ for large contaminations. This is no surprise: when outliers are hard to identify, the estimators~$\hat{V}_\gamma$, with~$\gamma\approx 1-\eta$, are likely to be based on some outliers, which will strongly affect the estimation performance.  

In this framework, Tyler shape depth, as announced, may be very useful to select a suitable value of~$\gamma$. We suggest choosing~$\gamma$ based on visual inspection of the curve~$\mathcal{C}=\{(\gamma, D(\hat{V}_\gamma,P_{\gamma,n})):\gamma\in[0.5,1]\}$, where $P_{\gamma,n}$ denotes the empirical measure associated with the optimal subsample leading to~$\hat{V}_\gamma$. The rationale is the following: for~$\gamma$ small, $D(\hat{V}_\gamma,P_{\gamma,n})$ will remain relatively high as long as no outlier is added to the optimal subsample. As~$\gamma$ increases and outliers are added in the computation of~$\hat{V}_\gamma$, the depth~$D(\hat{V}_\gamma,P_{\gamma,n})$ will sharply decrease, thereby forming a kink in~$\mathcal{C}$. The selected~$\gamma$ for a given dataset, $\hat\gamma$, should therefore be the largest value for which~$\mathcal{C}$ exhibit a stable behaviour. Figure~\ref{Fig2} plots the curve~$\mathcal{C}$ for the values of~$\delta$ and~$\eta$ considered above and clearly illustrates the behaviour of the depth curves just described. When the outliers are easily identifiable, the kinks occur at~$\gamma_0$, which coincides with~$1-\eta$. In the harder case, where outliers and clean data tend to be mixed, the selected value~$\hat\gamma$ is still remarkably close to~$\gamma_0$. 
In conclusion, Tyler shape depth, and the ranking of shape matrices it provides, yield an effective visual tool that allows the selection of a sensible trimming proportion~$1-\gamma$ in a data-driven way when conducting, e.g., a principal component analysis.

\subsection{Outlier detection}
\vspace{0mm}
\label{SecAppliOutlier}

For each trading day between February 1st, 2015 and February 1st, 2017, we collected the Nasdaq Composite and S$\&$P500 stock indices every five minutes and computed their returns, that is, the differences between two logs of consecutive index values. The returns on a given day form a bivariate dataset of usually $78$ observations, though the number of observations varies due to missing values; days with fewer than $70$ bivariate returns were discarded. The resulting dataset comprises $n=38489$ observations on $D=478$ trading days.

Our analysis studies the joint behaviour of the bivariate returns in order to determine which trading days are atypical. An important source of atypicality is associated with the overall scale of the bivariate returns, which alternate between periods of high and low volatility. Such deviations can easily be detected by comparing the trace of any scatter measure on intraday data with that on the whole dataset, so we focus instead on detecting atypical joint volatility, i.e., days on which the ratios of the marginal volatilities or the correlations between the returns deviate greatly from their  global behaviour.  

Let $\hat V_{\rm full}=\hat{V}_{\hat\gamma}$ denote the minimum covariance determinant shape estimator  computed from the full collection of~$n$ returns with maximal shape depth. More precisely, denoting as~$P_{\rm full}$ the empirical distribution of the full collection of returns, let $\hat\gamma=\arg \max_{\gamma\in\Gamma} D(\hat V_{\gamma},P_{\rm full})$, for $\Gamma=\{0.5,0.505,0.51,\dots, 0.995,1\}$. The value obtained is $\hat\gamma=0.825$, with corresponding depth $D(\hat{V}_{\hat\gamma},P_{\rm full})=0.497$. This high depth value ensures that~$\hat V_{\rm full}$ is an excellent proxy for the deepest shape matrix~$\hat{V}=\arg\max_V D(V,P_{\rm full})$, so the computation of~$\hat{V}$ is unnecessary. Returns at the beginning of each trading period are known to be more volatile and should be discarded in shape estimation, so the robustness of~$\hat V_{\rm full}$ is an obvious asset: the value of~$\hat\gamma$ allows us to adaptively discard days on which the volatility deviates from its global pattern. The procedure discarded more than half of the corresponding intra-day returns for~17 days, and, remarkably, $13$ of these days lie within the two atypical periods mentioned in the next paragraph.


For each day~$d=1,\dots,D$, we evaluated the depth~$D(\hat V_{\rm full},P_d)$ of the global shape estimate with respect to the empirical distribution~$P_d$ of the bivariate returns on day~$d$. 
The left panel of Figure~\ref{Fig3} presents the depth values~$D(\hat V_{\rm full},P_d)$. Vertical lines mark major events affecting the shape of the  volatility, while the two greyed rectangles cover two periods during which the markets notoriously gave atypical returns: the first period follows the devaluation of the Yuan on August 11th, 2015 which saw rapid changes in the stock markets, including large devaluations on August 24th, event~(a). The second period covers the beginning of 2016, when a slump in oil prices made stocks relying on oil very volatile compared to others. This resulted in atypical shape behaviour during January 22 -- February 9; this last day, event~(b), had the sharpest loss for the S$\&$P500 index. The other events are (c) the decision of the European Central Bank on March 10th, 2016 to extend quantitative easing thereby slashing interest rates, which had a significant positive impact on both the Nasdaq and S$\&$P500, but more pronounced for the latter, (d) the positive impact on the financial stocks following Fed officials' comments on the possibility of rate hike made on May 27, 2016, and (e) the aftermath of Donald Trump's election on November 9th. Detection of atypical observations was achieved by flagging outliers with a depth so low that it is outside the box-and-whiskers plot. This resulted in 12 flagged days, each either  being one of the events described above or lying in one of the greyed regions. 



We also computed the halfspace shape depth ${\rm HD}
(\hat V_{\rm full},P_d)$ of the global estimate  for each day~$d$ (\citealp{PVB17}). The right panel of Figure~\ref{Fig3}, a plot of $D(\hat V_{\rm full},P_{d})$ versus ${\rm HD}(\hat V_{\rm full},P_d)$, shows a clear positive association.
Halfspace shape depth values seem to have a higher concentration than Tyler's, because the former maximizes a concept of scatter depth in scale and may be able to find scatter estimates better suited to the data.
Indeed, a decrease in volatility in one of the marginals might be balanced by considering a scatter with a smaller scale which would have a large depth value. 
A byproduct of this is the fact that, when evaluating halfspace shape depth, the difficult maximisation step in scale seems to be crucial in correctly computing the depth ranking of the data, which can be affected by small deviations. 
More importantly, while events (a) and (b) receive low depth with respect to both concepts, only Tyler shape depth succeeds in flagging days associated with events (c) to (e) 
as outlying.

\section{Hypothesis testing for shape}
\vspace{0mm}
\label{Sectest}

In the previous section, we presented two specific applications of shape depth. The concept also allows us to tackle more standard inference problems for shape, such as point estimation and hypothesis testing. Here, we consider testing~$\mathcal{H}_{0}:V=V_0$ against $\mathcal{H}_{1}:V\neq V_0$ at level~$\alpha\in(0,1)$, where $V_0\in\mathcal{P}_{k,{\rm tr}}$ is fixed, based on a random sample~$X_1,\ldots,X_n$ from a $k$-variate elliptical distribution with known location~$\theta$ and unknown shape~$V$. 
In view of Theorem~\ref{TheorFishconst}, a natural depth-based test, $\phi_D$ say, rejects the null for small values of~$T_{\theta,n}=D_\theta(V_0,P_n)$, where $P_n$ is the empirical distribution of~$X_1,\ldots,X_n$. Since~$T_{\theta,n}$ is discrete, achieving null size~$\alpha$ in general requires randomization. The resulting test thus rejects the null hypothesis if
$
T_{\theta,n}<t_{\alpha,n}
$,
rejects the null hypothesis with probability~$\gamma_{\alpha,n}$ if 
$
T_{\theta,n}=t_{\alpha,n}
$,
and does not reject the null hypothesis if 
$
T_{\theta,n}>t_{\alpha,n}
$, 
where~$t_{\alpha,n}$ is the null $\alpha$-quantile of~$T_{\theta,n}$ and $\gamma_{\alpha,n}$ is the amount of randomization. Under the assumption that~$P$ does not charge the center of the distribution, $T_{\theta,n}$ is distribution-free under the null hypothesis, which allows estimating~$t_{\alpha,n}$ and~$\gamma_{\alpha,n}$ arbitrarily well through simulations. Prior to applying the test below for~$k=2$ at level~$5\%$ with sample sizes~$n=200$, $500$, these were estimated from~$500,\!000$ mutually independent standard normal samples for each sample size, yielding~$\hat{t}_{0.05,200}=0.40$, $\hat{\gamma}_{0.05,200}=0.61$, $\hat{t}_{0.05,500}=0.43$ and~$\hat{\gamma}_{0.05,500}=0.25$. Distribution-freeness of~$T_{\theta,n}$ under the null hypothesis actually extends to the class of distributions with elliptical directions (\citealp{Ran2000}). 

We performed two simulations in the bivariate case. The first considers the problem of testing the null hypothesis of sphericity~$\mathcal{H}_0: V_0=I_2$ about~$\theta=0$ and compares the finite-sample powers of~$\phi_D$ with those of some competitors. For each value of~$\ell=0,1,\ldots,6$ we generated $M=3,\!000$ independent random samples~$X_1,\ldots,X_n$ of size~$n=500$ from the normal distribution with location~$\theta=0$ and shape
$$
V_{\ell,\xi}
=
I_2+\ell\xi
{1\ 0.5\, \choose \,0.5\ -1\,}
$$
and from the corresponding elliptical Cauchy distribution. The value~$\ell=0$ corresponds to the null hypothesis, whereas $\ell=1,\ldots,6$ provide increasingly severe alternatives. We took~$\xi=0.035$ and $0.045$ for the normal and Cauchy samples in order to obtain roughly the same rejection frequencies in both cases. 

For each sample, we carried out six tests at nominal level~$5\%$: 
(i) the test~$\phi_D$ described above; 
(ii) the Gaussian test from \cite{Joh1972}, or more precisely, its extension to elliptical distributions with finite fourth-order moments from \cite{HalPai2006B}; 
(iii) the sign test from \cite{HalPai2006B}; 
(iv) the Wald test based on the \cite{Tyl1987} scatter matrix;
(v)--(vi) the tests from~\cite{PaiVanB2014} based on the 
shape estimator~$\hat{V}_\gamma$ in $\S$~\ref{Secapplis}, with~$\gamma=0.5$ and~$\gamma=0.8$.
The tests (ii)--(vi) were performed based on their asymptotic null distribution. 
The rejection frequencies in Figure~\ref{Fig4} reveal that~$\phi_D$ performs very similarly to, although it may be slightly dominated by, the sign-based tests in~(iii)--(iv) but performs very well under heavy tails, where it beats all other tests. As expected, the Gaussian test collapses under heavy tails and the minimum covariance determinant tests show low empirical power.

  The second simulation tests~$\mathcal{H}_{0}:V=V_0$, with $V_0={\rm diag}(2,1/2)$ and specified location~$\theta=0$, and compares the tests above in terms of the level robustness (\citealp{Heetal1990}). We considered mixture distributions~$P^{X_{(\eta)}}=(1-\eta)P^{X}+\eta P^{Y}$ with several contamination levels~$\eta$. Here, $X$ is a bivariate, normal or elliptical Cauchy, null random vector. The contamination random vector~$Y$ was chosen as follows: 
(a) 
$Y$ has the same distribution as the vector obtained by rotating~$X$ about the origin by~$45$ degrees;
(b) 
$Y$ has the same elliptical distribution as~$X$ but its shape is $V=I_2$;
(c) 
$Y$ is obtained by multiplying the vector~$Y$ in~(b) by four. 
The uncontaminated distribution~$P^X$ puts more mass along the horizontal axis. In~(a), the contamination typically shows along the main bisector, whereas the contamination in~(b) is uniformly distributed over the unit circle. As for~(c), the contamination combines the directional feature of~(b) with radial outlyingness. 
For each combination of distribution, normal or Cauchy, of contamination pattern, (a)--(c), and of contamination level, $\eta=0,0.025,0.05,0.1,0.2,0.25$ or $0.3$, we generated $3,\!000$ independent random samples~$X_{(\eta)1},\ldots,X_{(\eta)n}$ of size~$n=200$. Figure~\ref{Fig5} plots the resulting rejection frequencies and reveals the very good robustness of the depth-based test~$\phi_D$; recall that, irrespective of~$\eta$, the target rejection frequency is here~$5\%$. In particular, $\phi_D$ always dominates its sign-based competitors~(iii)--(iv). The minimum covariance determinant tests~(v)--(vi) dominate~$\phi_D$ in terms of robustness but  exhibit poor finite-sample power. Radial outliers strongly affect the Gaussian test.

Summing up, the test associated with the proposed shape depth provides a good balance between efficiency and robustness. The improved robustness compared to its sign-based competitors is obtained at a very slight loss of power. Depth-based procedures can thus be defined for standard inference problems on shape, and will tend to perform as well as sign-based procedures. As shown in $\S$~\ref{Secapplis}, however, shape depth provides a whole ranking of shape matrices that allows addressing less standard applications.

\section{Perspectives for future research}
\vspace{0mm}
\label{Secfinal}

The present work offers quite rich research perspectives. The asymptotic distributions of the sample depths~$D_\theta(V,P_n)$ and~$D(V,P_n)$ as well as those of the corresponding deepest shape estimators could be studied. Investigating the robustness properties of these shape estimators would also be of interest, in particular to see whether these estimators have a high breakdown point. Regarding hypothesis testing, it would be desirable to define depth-based tests for other shape problems, such as testing the null hypothesis that two populations share the same shape. 

Another key point is related to computational aspects. Since Tyler shape depth was defined through halfspace depth, it can in principle be evaluated by using the numerous packages that are dedicated to halfspace depth. The definition of Tyler shape depth suggests that evaluation of this depth in dimension~$k$ requires the computation of halfspace depth in dimension~$k^2$. Fortunately, redundancies in the random vector~$W_{\theta,V}$ reduce the dimension from~$k^2$ to~$d_k=k(k+1)/2-1$ as shown by the following result. 

\begin{theorem}
\label{dimD}
Let~$P=P^X$ be a probability measure over~$\R^k$ and fix~$\theta\in\R^k$. 
Let~${\rm vech}(A)$ be the vector stacking the lower-diagonal entries of~$A=(A_{ij})$ on top of each other and~${\rm vech}_0\,A$ be~${\rm vech}(A)$ deprived of its first component. Then,  
$D_\theta(V,P)
\!
=
\!
D(0,P^{\tilde{W}_{\theta,V}}) 
\!
=
\!
 \inf_{u\in\mathcal{S}^{d_k-1}}
\!
 {\rm pr}(
  u^\T \tilde{W}_{\theta,V}
 \geq 0 
)
$, with~$
\tilde{W}_{\theta,V}
\!
=
\!
{\rm vech}_0\{ U_{\theta,V} U^\T_{\theta,V} - (1/k) I_k \}
$.
\end{theorem}

It follows that, for~$k=2$ and~$3$, Tyler shape depth dominates  its halfspace counterpart from \cite{PVB17} from a computational point of view. There is, though, probably room for ad hoc algorithms to compute Tyler shape depth more efficiently. It would also be desirable to design iterative algorithms for the computation of deepest shape matrices.



%
\appendix

\section{Appendix}

As in the main manuscript, $\rm pr$ will refer to probability under the probability measure~$P$ at hand. However, it will sometimes be needed to emphasize the underlying probability measure, in which case we will write~${\rm pr}_P$, ${\rm pr}_Q$, ${\rm pr}_{P_n}$, etc. 

Many of the subsequent results require the following lemma.

\begin{lemma}
	\label{LemrewritingD}
Let~$P$ be a probability measure over~$\R^k$ and fix~$\theta\in\R^k$. Write
$
C^M_{\theta,V}
 =
 \big\{ x \in \R^{k}\setminus\{\theta\} : 
 (u_{\theta,V}^x)^\T M u_{\theta,V}^x 
 \geq  {{\rm tr}(M)/k}  
 \big\}
$
and
$\tilde{C}^M_{\theta,V}
 =
 \big\{ x \in \R^{k} : 
 (u_{\theta,V}^x)^\T M u_{\theta,V}^x 
 \geq  {{\rm tr}(M)/k}  
 \big\}
,
 $
where~$u_{\theta,V}^x$ is defined as~$V^{-1/2}(x-\theta)/\|V^{-1/2}(x-\theta)\|$ if~$x\neq \theta$ and as~$0$ otherwise. 
Then, for any~$V\in\mathcal{P}_{k,{\rm tr}}$ and any~$r\in\R$,
$$
D_\theta(V,P)
=
 \inf_{M\in\mathcal{M}^{\rm all}_k}
 {\rm pr}\big( \tilde{C}^M_{\theta,V} \big)
=
 \inf_{M\in\mathcal{M}^{\rm all}_{k,F}}
 {\rm pr}\big( \tilde{C}^M_{\theta,V} \big)
=
 \inf_{M\in\mathcal{M}^{\rm all}_k}
 {\rm pr}\big(C^M_{\theta,V} \big)
=
 \inf_{M\in\mathcal{M}^r_k}
 {\rm pr}\big( C^M_{\theta,V} \big)
 ,
$$ 
where~$\mathcal{M}^{\rm all}_k$ collects the $k\times k$ symmetric matrices with arbitrary trace, $\mathcal{M}^r_k$ is the subset of~$\mathcal{M}^{\rm all}_k$ of matrices with trace~$r$, and where~$\mathcal{M}^{\rm all}_{k,F}$ is the collection of matrices in~$\mathcal{M}^{\rm all}_k$ with Frobenius norm one. 
\end{lemma}


\begin{proof}
It directly follows from the definition of Tyler shape depth that
$$
	D_\theta(V,P) 
=
 \inf_{v\in\R^{k^2}}
{\rm pr}\big[ \big\{ x \in \R^{k} : 
 v^\T 
{\rm vec}\,\{u_{\theta,V}^x(u_{\theta,V}^x)^\T- {(1/k)} I_k\}
 \geq 0  
 \big\} \big]
 .
$$
When~$v$ runs over~$\R^{k^2}$, the matrix~$M$ satisfying $v={\rm vec}(M^\T)$ runs over the collection~$\mathcal{N}_k$ of \mbox{$k\times k$} matrices. Since~$(u_{\theta,V}^x)^\T M u_{\theta,V}^x=(u_{\theta,V}^x)^\T \{(M+M^\T)/2\} u_{\theta,V}^x$ for any~$M\in\mathcal{N}_k$, this yields 
\begin{eqnarray}
D_\theta(V,P) 
&=&
 \inf_{M\in\mathcal{N}_k}
 {\rm pr}\big[ \big\{ x \in \R^{k^2} : 
 {\rm tr}\big[ M \{u_{\theta,V}^x(u_{\theta,V}^x)^\T- {(1/k)} I_k\} \big]
 \geq 0 
 \big\} \big]
 \nonumber
\\[2mm]
&=&
 \inf_{M\in\mathcal{N}_k}
 {\rm pr}\big(
\tilde{C}_{\theta,V}^M 
 \big)
=
 \inf_{M\in\mathcal{M}_k^{\rm all}}
 {\rm pr}\big( 
\tilde{C}_{\theta,V}^M 
 \big)
.
\label{altdepthconst}
\end{eqnarray}
Letting~$\mathbb{I}(A)$ be equal to one if condition~A holds and to zero otherwise, this provides
\begin{equation}
D_\theta(V,P) 
=
 \inf_{M\in\mathcal{M}_k^{\rm all}}
 \Big(
 {\rm pr}\big( 
C_{\theta,V}^M 
 \big)
+
{\rm pr}[\{\theta\}] 
 \mathbb{I}\big\{
 {\rm tr}(M) \leq 0
  \big\}
  \Big)
=
 \inf_{M\in\mathcal{M}_k^{\rm all}}
{\rm pr}\big( 
C_{\theta,V}^M 
 \big)
,
\label{altdepthconst2}
\end{equation}
where we have used the fact that~${\rm pr}\big(C_{\theta,V}^M \big)$ is unchanged when~$M$ is replaced with~$M+\lambda I_k$ for any~$\lambda\in\R$. The same invariance property explains that the infimum over~$\mathcal{M}_k^{\rm all}$ in~(\ref{altdepthconst2}) may be replaced with an infimum over~$\mathcal{M}_k^r$ for any~$r$. Finally, the result for~$\mathcal{M}^{\rm all}_{k,F}$ follows from~(\ref{altdepthconst}) by noting that~$\tilde{C}_{\theta,V}^{\lambda M}=\tilde{C}_{\theta,V}^M$ for any~$\lambda>0$ and that~$M=0$ cannot provide the infimum in~(\ref{altdepthconst}). 
The proof is complete. 
\end{proof}
\vspace{0mm}


\begin{proof}[Proof of Theorem~\ref{Theorcontinuity}]
(i) Fix~$M\in\mathcal{M}^{\rm all}_k$ and consider~$
\tilde{C}^M
=
\tilde{C}^M_{0,I_k}
$,
where~$\tilde{C}^M_{\theta,V}$ was defined in Lemma~\ref{LemrewritingD}. 
Since~$\tilde{C}^M$ is closed, the mapping~$P\mapsto {\rm pr}_P( \tilde{C}^M )$ is upper semicontinuous for weak convergence. Now, Slutzky's lemma entails that, as~$d(V,V_0)\to 0$, the measure defined by~$B\mapsto {\rm pr}( \theta+V^{1/2}B)$ converges weakly to the one defined by~$B\mapsto {\rm pr}(\theta+ V_0^{1/2}B)$. Therefore,~$V \mapsto {\rm pr}( \theta+V^{1/2} \tilde{C}^M)={\rm pr}(\tilde{C}^M_{\theta,V})$ is upper semicontinuous at~$V_0$. From Lemma~\ref{LemrewritingD}, we then obtain that
$$
V
\mapsto
D_\theta(V,P)
=
\inf_{M\in\mathcal{M}^{\rm all}_k} 
{\rm pr}( \tilde{C}^M_{\theta,V} )
,
$$
is upper semicontinuous, as it is the infimum of a collection of upper semicontinuous functions. 
(ii) The result follows from the fact that the depth region~$R_{\theta}(\alpha,P)$ is the inverse image of~$[\alpha,\infty)$ by the upper semicontinuous function~$V\mapsto D_\theta(V,P)$. 
(iii) Fix a sequence~$(V_n)$ in~$\mathcal{P}_{k,{\rm tr}}$ such that~$d(V_n,V_0)\to 0$. In view of Lemma~\ref{LemrewritingD} again, we can, for any~$n$, pick~$M_n\in\mathcal{M}^{\rm all}_{k,F}$ such that 
$
P( \tilde{C}^{M_n}_{\theta,V_n} )
\leq 
D_\theta(V_n,P)+\frac{1}{n}
$.
Compactness of~$\mathcal{M}^{\rm all}_{k,F}$ ensures that we can extract a subsequence~$(M_{n_\ell})$ of~$(M_n)$ that converges to~$M_0\in\mathcal{M}^{\rm all}_{k,F}$. Writing~$\mathbb{I}(B)$ for the indicator function of the set~$B$, the dominated convergence theorem then yields that
$$
{\rm pr}( \tilde{C}^{M_{n_\ell}}_{\theta,V_{n_\ell}} )
-
{\rm pr}(\tilde{C}^{M_{0}}_{\theta,V_{0}} )
=
\int_{\R^k} 
\big\{
\mathbb{I}( \tilde{C}^{M_{n_\ell}}_{\theta,V_{n_\ell}} )
-
\mathbb{I}( \tilde{C}^{M_{0}}_{\theta,V_{0}} )
\big\}
\,dP
\to 0
$$
as~$\ell\to\infty$. The absolute continuity assumption on~$P$ guarantees that~$\mathbb{I}( \tilde{C}^{M_{n_\ell}}_{\theta,V_{n_\ell}} )
-
\mathbb{I}( \tilde{C}^{M_{0}}_{\theta,V_{0}} )
\to 0$ $P$-almost everywhere. Consequently,
$$
\liminf_{n\to\infty} D_\theta(V_n,P)
=
\liminf_{n\to\infty} 
{\rm pr}\big( \tilde{C}^{M}_{\theta,V_{n}} \big)
= 
\liminf_{\ell\to\infty} 
{\rm pr}\big( \tilde{C}^{M_{n_\ell}}_{\theta,V_{n_\ell}} \big)
=
{\rm pr}\big(\tilde{C}^{M_{0}}_{\theta,V_{0}} \big)
\geq
D_\theta(V_0,P)
.
$$
We conclude that, if~$P$ is absolutely continuous with respect to the Lebesgue measure, then~$V\mapsto D_\theta(V,P)$ is also lower semicontinuous, hence continuous.
\end{proof}
\vspace{0mm}


The proof of Theorem~\ref{Theorboundedness} requires the following result.

 \begin{lemma}
\label{lemtp}
Let~$P$ be a probability measure over~$\R^k$ and fix~$\theta\in\R^k$. 
Write~$u_{\theta}^x=(x-\theta)/\|x-\theta\|$ if~$x\neq \theta$ and 0 otherwise.
For any~$c\geq 0$, further let
 $
 t_{\theta,P}(c)=\sup_{v\in \mathcal{S}^{k-1}} {\rm pr}( |v^\T u_{\theta}^X| \leq  c )$, so that~$t_{\theta,P}=t_{\theta,P}(0)=\sup_{v\in \mathcal{S}^{k-1}} {\rm pr}\{ v^\T (X-\theta)=0 \}$. Then,
$t_{\theta,P}(c)\to t_{\theta,P}$  as~$c\to 0$.
\end{lemma}
\vspace{0mm}

\begin{proof}[Proof of Lemma~\ref{lemtp}]
Since~$t_{\theta,P}(c)$ is increasing in~$c$ over~$[0,\infty)$ and is larger than or equal to~$t_{\theta,P}$ for any positive~$c$, we have that
$
\tilde{t}_{\theta,P}
=
\lim_{c\to 0} 
t_{\theta,P}(c)
$
exists and is such that~$\tilde{t}_{\theta,P}\geq t_{\theta,P}$. Now, fix a decreasing sequence~$(c_n)$ converging to~$0$ and consider an arbitrary sequence~$(v_n)$ such that
$$
{\rm pr}( |v_n^\T u_{\theta}^X|\leq c_n ) 
\geq
t_{\theta,P}(c_n) - (1/n)
.
$$
Since~$\mathcal{S}^{k-1}$ is compact, we can consider a subsequence~$(v_{n_\ell})$ that converges to~$v_0\in\mathcal{S}^{k-1}$; without loss of generality, we can of course assume that this subsequence is such that~$(v_0^\T v_{n_\ell})$ is an increasing sequence.
Let then~$C_\ell=\{ v\in\mathcal{S}^{k-1} : v_0^\T v \geq v_0^\T v_{n_\ell} \}$. Clearly, $C_\ell$ is a decreasing sequence of sets with $\cap_\ell C_\ell=\{v_0\}$, so that 
$$
\lim_{\ell\to\infty} 
{\rm pr}\big[ u_{\theta}^X\in \cup_{v\in C_\ell} \{ y:|v^\T y|\leq c_{n_\ell} \} \big]
=
{\rm pr}\big[ u_{\theta}^X\in  \{ y:|v_0^\T y|\leq 0 \} \big] 
=
{\rm pr}\big( |v_0^\T u_{\theta}^X|=0 \big) 
.
$$  
Now, for any~$\ell$, 
we have 
$
{\rm pr}\big[ u_{\theta}^X\in \cup_{v\in C_\ell} \{y: |v^\T y|\leq c_{n_\ell} \} \big] 
\geq  
{\rm pr}( |v_{n_\ell}^\T u_{\theta}^X|\leq c_{n_\ell} ) 
\geq 
t_{\theta,P}(c_{n_\ell})-(1/n_\ell)
$,
which implies that~$t_{\theta,P} \geq {\rm pr}( |v_0^\T u_{\theta}^X|=0 )  \geq \tilde{t}_{\theta,P}$. 
\end{proof}
\vspace{0mm}


\begin{proof}[Proof of Theorem~\ref{Theorboundedness}]
Fix~$V\in\mathcal{P}_{k,{\rm tr}}$ and denote as~$\lambda_1(V)$ the largest eigenvalue of~$V$. Similarly, denote $\lambda_k(V)$. Possible ties are unimportant below. Letting~$v_1(V)$ and~$v_k(V)$ be arbitrary corresponding unit eigenvectors,  Lemma~\ref{LemrewritingD} provides, with~$M_V=v_1(V) \linebreak 
v_1^\T(V)\in \mathcal{M}^{\rm all}_k$
\begin{eqnarray*}
\hspace{-1mm} 
D_\theta(V,P) 
&\leq &
 {\rm pr}\Big\{ 
 (u_{\theta,V}^X)^\T M_{V} u_{\theta,V}^X 
 \geq  {\rm tr}(M_{V})/k
 ,
 X\neq \theta 
 \Big\}
\\[2mm]
& = &
 {\rm pr}\Big[ 
k \lambda^{-1}_1(V) \{v_1^\T(V)(X-\theta)\}^2 \geq 
\|V^{-1/2}(X-\theta)\|^2
 ,
 X\neq \theta 
 \Big]
\\[2mm]
& \leq &
 {\rm pr}\Big\{ 
\|V^{-1/2}(X-\theta)\|^2 \leq k \|X-\theta\|^2 
 ,
 X\neq \theta 
 \Big\}\\[2mm]
&=& 
 {\rm pr}\big( 
\|V^{-1/2}u_{\theta,I_k}^X\|^2 \leq k 
 ,
 u_{\theta}^X\neq 0
 \big)
,
\end{eqnarray*}
where we used the inequality~$\lambda_1(V)\geq 1$ which follows from the constraint~${\rm tr}(V)=k$, and where~$u_{\theta}^s$ is defined in Lemma~\ref{lemtp}.
Therefore, 
\begin{equation}
\label{tousebound}
D_\theta(V,P) 
 \leq 
 {\rm pr}\big[ 
 \lambda_1(V^{-1}) \{v_1^\T(V) u_{\theta}^X\}^2
\leq k 
 \big]
\leq
 t_{\theta,P}\big[ \{k \lambda_k(V)\}^{1/2} \big]
.
\end{equation}
Now, ad absurdum, take~$\varepsilon>0$ such that~$R_\theta(t_{\theta,P}+\varepsilon,P)$ is unbounded. This implies that there exists a sequence~$(V_n)$ in~$\mathcal{P}_{k,{\rm tr}}$ satisfying~$D_\theta(V_n,P)\geq t_{\theta,P}+\varepsilon$ for any~$n$ and for which~$d(V_n,I_k)\to \infty$. Since~$\lambda_1(V_n)<{\rm tr}(V_n)=k$, we must have that~$\lambda_k(V_n)\to 0$. Lemma~\ref{lemtp} and~(\ref{tousebound}) then imply that $D_\theta(V_n,P)<t_{\theta,P}+\varepsilon$ for~$n$ large enough, a contradiction. Consequently, $R_\theta(\alpha,P)$ is bounded for any~$\alpha>t_{\theta,P}$. 

Now, Lemma~C.1 in \cite{PVB17} readily implies that a bounded subset of~$\mathcal{P}_{k,{\rm tr}}$ is also totally bounded, in the sense that, for any~$\varepsilon>0$, it can be covered by finitely many balls of the form~$B(V,\varepsilon)=\{ \tilde{V}\in\mathcal{P}_{k,{\rm tr}} : d(\tilde{V},V)<\varepsilon\}$. Part~(i) of the result and Theorem~\ref{Theorcontinuity}(ii) thus entail that, for any~$\alpha>t_{\theta,P}$, the region~$R_{\theta}(\alpha,P)$ is closed and totally bounded. The result then follows from the completeness of the metric space~$(\mathcal{P}_{k,{\rm tr}},d)$.
  \end{proof} 
\vspace{0mm} 
 

\begin{proof}[Proof of Theorem~\ref{Theorexistmax}]
Let $\alpha_*=\sup_{V\in\mathcal{P}_{k,{\rm tr}}} D_\theta(V,P)$. By assumption,~$R_\theta(t_{\theta,P},P)$ is non-empty. Thus, $\alpha_*\geq t_{\theta,P}$ and the result holds if~$\alpha_*=t_{\theta,P}$. We may therefore assume that~$\alpha_*>t_{\theta,P}$. For any~$n$, pick then~$V_n$ in~$R_\theta(\alpha_*-1/n,P)$, where~$R_\theta(\alpha,P)$ is defined as~$\mathcal{P}_{k,{\rm tr}}$ for~$\alpha<0$. Fix~$\varepsilon\in(0,\alpha_*-t_{\theta,P})$. For~$n$ large enough, all terms of the sequence~$(V_n)$ belong to the compact set~$R_\theta(\alpha_*-\varepsilon,P)$; see Theorem~\ref{Theorboundedness}. Thus, there exists a subsequence~$(V_{n_k})$ that converges in~$R_\theta(\alpha_*-\varepsilon,P)$, to~$V_*$ say. For any~$\varepsilon'\in(0,\varepsilon)$, all~$(V_{n_k})$ eventually belong to the closed set~$R_\theta(\alpha_*-\varepsilon',P)$, so that~$V_*\in R_\theta(\alpha_*-\varepsilon',P)$. Therefore, $\alpha_*-\varepsilon'\leq D_\theta(V_*,P)\leq \alpha_*$ for any such~$\varepsilon'$, which establishes the result. 
\end{proof}
\vspace{0mm}


The proof of Theorem~\ref{TheorQuasic} requires the following preliminary result.

\begin{lemma}
For any~$y\in\R^k$ and any $k\times k$ symmetric matrix~$M$, the mapping~$V\mapsto {\rm tr}(MV) y^\T V^{-1} y$ is quasi-convex, that is, for any~$V_a,V_b\in\mathcal{P}_{k,{\rm tr}}$ and any~$t\in [0,1]$, ${\rm tr}(MV_t) y^\T V_t^{-1} y \leq \max\{{\rm tr}(MV_a) y^\T V^{-1}_a y,{\rm tr}(MV_b) y^\T V^{-1}_b y\}$, with~$V_t=(1-t)V_a+t V_b$. 
\label{quasiconvexe}	
\end{lemma}

\begin{proof}
	We treat two cases separately. (i) Assume first that~${\rm tr}(MV_a){\rm tr}(MV_b)>0$. Write
	$$
	\frac{V_t}{{\rm tr}(MV_t)}
	=
	(1-s_t) \frac{V_a}{{\rm tr}(MV_a)} + s_t \frac{V_b}{{\rm tr}(MV_b)}
,
\quad \textrm{with }
s_t
=
\frac{t\, {\rm tr}(MV_b)}{(1-t){\rm tr}(MV_a)+t\, {\rm tr}(MV_b)}
\cdot
	$$ 
Since~$s_t\in[0,1]$, the weighted harmonic-arithmetic matrix inequality then shows that, for any~$y\in\R^k$, 
\begin{eqnarray*}
y^\T \bigg\{ \frac{V_t}{{\rm tr}(MV_t)} \bigg\}^{-1} y
&\leq &
y^\T \bigg[ 	(1-s_t) \bigg\{\frac{V_a}{{\rm tr}(MV_a)}\bigg\}^{-1} + s_t \bigg\{\frac{V_b}{{\rm tr}(MV_b)}\bigg\}^{-1} \bigg] y
\\[2mm]
&\leq &
\max\bigg[ y^\T\bigg\{\frac{V_a}{{\rm tr}(MV_a)}\bigg\}^{-1}y, y^\T\bigg\{\frac{V_b}{{\rm tr}(MV_b)}\bigg\}^{-1}y\bigg]
,
\end{eqnarray*}
as was to be showed; we refer to Lemma~2.1(vii) in \citealp{LawLim13} for the aforementioned inequality.
(ii) Assume then that ${\rm tr}(MV_a){\rm tr}(MV_b)\leq 0$. Without loss of generality, assume that~${\rm tr}(MV_a)\leq 0$ and~${\rm tr}(MV_b)\geq 0$. If~${\rm tr}(MV_a)={\rm tr}(MV_b)=0$, then~${\rm tr}(MV_t)=0$ for any~$t$ and the result trivially holds. Hence, we may assume that~${\rm tr}(MV_a)\neq 0$ or~${\rm tr}(MV_b)\neq 0$, which implies that~${\rm tr}(MV_{t_0})=0$ for a unique~$t_0\in[0,1]$. From continuity, pick then~$\delta\in(0,1-t_0)$ such that, for any~$t\in [t_0,t_0+\delta)$,
\begin{eqnarray*}
	{\rm tr}(MV_t) y^\T V^{-1}_t y 
	&\leq &
	{\rm tr}(MV_b) y^\T V^{-1}_b y
\\[2mm]
&\leq &
\max\big\{{\rm tr}(MV_a) y^\T V^{-1}_a y,{\rm tr}(MV_b) y^\T V^{-1}_b y\big\}
.
\end{eqnarray*}
By applying Part~(i) of the proof with~$V_{t_0+\delta}$ and~$V_b$, we obtain that, for any~$t\in[t_0+\delta,1]$, 
\begin{eqnarray*}
{\rm tr}(MV_t) y^\T V^{-1}_t y
&\leq &
\max\big\{{\rm tr}(MV_{t_0+\delta}) y^\T V^{-1}_{t_0+\delta} y,{\rm tr}(MV_b) y^\T V^{-1}_b y\big\}
\\[2mm]
&\leq &
\max\big\{{\rm tr}(MV_a) y^\T V^{-1}_a y,{\rm tr}(MV_b) y^\T V^{-1}_b y\big\}
.
\end{eqnarray*} 
Since~${\rm tr}(MV_t) y^\T V^{-1}_t y\leq 0
\leq \max\big\{{\rm tr}(MV_a) y^\T V^{-1}_a y,{\rm tr}(MV_b) y^\T V^{-1}_b y\big\}
$ for any~$t\in[0,t_0]$, the result follows.
\end{proof}
\vspace{0mm}
 

\begin{proof}[Proof of Theorem~\ref{TheorQuasic}]
(i) Write~$V_t=(1-t)V_a+t V_b$, where~$V_a,V_b\in\mathcal{P}_{k,{\rm tr}}$ and~$t\in [0,1]$ are fixed. First note that, letting~$d^2_\theta(V)=(X-\theta)^\T V^{-1} (X-\theta)$, Lemma~\ref{LemrewritingD} yields
\begin{eqnarray}
D_\theta(V,P)
&=&
 \inf_{M\in\mathcal{M}_k^{\rm all}}
 {\rm pr}\big\{ 
 (X-\theta)^\T V^{-1/2} M V^{-1/2} (X-\theta) 
 \geq  (1/k) {\rm tr}(M) d^2_\theta(V)
,X\neq \theta
 \big\}
\nonumber 
\\[2mm]
&=&
 \inf_{M\in\mathcal{M}_k^{\rm all}}
 {\rm pr}\big\{ 
 (X-\theta)^\T M (X-\theta) 
 \geq  (1/k) {\rm tr}(MV) d^2_\theta(V)
,X\neq \theta
 \big\}
.
\label{gskri}
\end{eqnarray} 
Writing again~$V_t=(1-t)V_a+t V_b$, Lemma~\ref{quasiconvexe} thus yields that, for any~$M\in \mathcal{M}^{\rm all}_k$,
\begin{eqnarray*}
\lefteqn{
	\hspace{2mm}  
	 {\rm pr}\big\{
 (X-\theta)^\T M (X-\theta) 
 \geq  (1/k) {\rm tr}(MV_t) d^2_\theta(V_t)
,X\neq \theta
 \big\}
}
\\[2mm] 
& &
\hspace{13mm} 
 \geq 
 {\rm pr}\big[ 
 (X-\theta)^\T M (X-\theta) 
 \geq  (1/k)  \max\{{\rm tr}(MV_a)d^2_\theta(V_a),{\rm tr}(MV_b)d^2_\theta(V_b)\}
,X\neq \theta
 \big]
\\[2mm]
& &
\hspace{13mm} 
 = 
\min\Big[
 {\rm pr}\big\{ 
 (X-\theta)^\T M (X-\theta) 
 \geq  (1/k) {\rm tr}(MV_a) d^2_\theta(V_a)
,X\neq \theta
 \big\}
,
\\[2mm]
& &
\hspace{33mm} 
 {\rm pr}\big\{ 
 (X-\theta)^\T M (X-\theta) 
 \geq  (1/k) {\rm tr}(MV_b) d^2_\theta(V_b)
,X\neq \theta
 \big\}
\Big]
 \\[2mm]
& &
\hspace{13mm} 
 \geq 
\min\{D_\theta(V_a,P),D_\theta(V_b,P)\}
.
\end{eqnarray*}
The result then follows from~(\ref{gskri}).
(ii) If~$V_a,V_b\in R_{\theta}(\alpha,P)$, then Part~(i) of the result entails that~$D_\theta(V_t,P) \geq \min\{D_\theta(V_a,P),D_\theta(V_b,P)\}\geq \alpha$, so that~$V_t\in R_{\theta}(\alpha,P)$. 
\end{proof}
\vspace{0mm}




The proof of Theorem~\ref{TheorFishconst} requires both following lemmas. 


\begin{lemma}
\label{LemConsist1}
Let~$P$ be elliptical over~$\R^k$ with location~$0$ and shape~$I_k$. 
Then, 
$
D_0(I_k,P)
=
(1-{\rm pr}[\{0\}])
{\rm pr}(U_1^2>1/k)
,
$
where~$U=(U_1,\ldots,U_k)^\T$ is uniformly distributed over the unit sphere $\mathcal{S}^{k-1}$.
\end{lemma}




\begin{lemma}
\label{LemConsist3}
Let~$P$ be elliptical over~$\R^k$ with location~$0$ and shape~$I_k$. 
   Then, for any $V\in\mathcal{P}_{k,{\rm tr}}\setminus \{I_k\}$, 
$
D_0(V,P)
<
(1-{\rm pr}[\{0\}])
{\rm pr}(U_1^2>1/k)
$,
where 
$U=(U_1,\ldots,U_k)^\T$ is uniformly distributed over~$\mathcal{S}^{k-1}$.
\end{lemma}

 

\begin{proof}[Proof of Lemma~\ref{LemConsist1}]
In the spherical setup considered, we have that, for any~$M\in \mathcal{M}_k^{\rm all}$, 
$$
{\rm pr}\big\{ 
(X^\T M X)/\|X\|^2
 \geq  {\rm tr}(M)/k 
 , \, X\neq 0
 \big\}
=
{\rm pr}\big\{ 
U^\T M U
 \geq  {\rm tr}(M)/k 
\big\}
{\rm pr}( X\neq 0 )
,
$$
where $U=(U_1,\ldots,U_k)^\T$ is uniform over~$\mathcal{S}^{k-1}$.  
 Lemma~\ref{LemrewritingD} then entails that
$$
D_0(I_k,P)
=
(1-{\rm pr}[\{0\}])
 \inf_{M\in\mathcal{M}_k^{\rm all}}
{\rm pr}\big\{ 
U^\T M U 
 \geq  {\rm tr}(M) /k
 \big\}
.
$$
Decomposing~$M$ into~$O\Lambda O^\T$, where~$O$ is a $k\times k$ orthogonal matrix and where~$\Lambda={\rm diag}(\lambda_1,\ldots,\lambda_k)$ is a diagonal matrix, this yields 
\begin{eqnarray*}
	\lefteqn{
D_0(I_k,P)
=
(1-{\rm pr}[\{0\}])
 \inf_{\lambda\in\R^k}
{\rm pr}\bigg( 
\sum_{\ell=1}^k \lambda_\ell U_\ell^2 
 \geq  k^{-1} \sum_{\ell=1}^k \lambda_\ell 
 \bigg)
}
\\[2mm]
& & 
\hspace{15mm} 
=
(1-{\rm pr}[\{0\}])
 \inf_{\lambda\in\R^k}
{\rm pr}\bigg\{ 
\sum_{\ell=1}^k \lambda_\ell \big(U_\ell^2 - k^{-1}\big)
 \geq 0
 \bigg\}
=
(1-{\rm pr}[\{0\}])
 \inf_{\lambda\in\R^k}
p(\lambda) 
.
\end{eqnarray*}
By using successively the facts that~$p(0)=1$ and~$p(\lambda)=p(\lambda/\|\lambda\|)$ for any~$\lambda\in\R^k\setminus\{0\}$, we obtain
\begin{equation}
	\label{lastinf}
	D_0(I_k,P)
=
(1-{\rm pr}[\{0\}])
 \inf_{\lambda\in\R^k\setminus\{0\}}
p(\lambda) 
=
(1-{\rm pr}[\{0\}])
 \inf_{\lambda\in\mathcal{S}^{k-1}}
p(\lambda) 
.
\end{equation}
The result then follows from Theorem~2 from \cite{PVB17b}, that states that the last infimum in~(\ref{lastinf}) is equal to~${\rm pr}(U_1^2>1/k)$.   
\end{proof}
\vspace{0mm}




\begin{proof}[Proof of Lemma~\ref{LemConsist3}]
Fix~$V\in \mathcal{P}_{k,{\rm tr}}$ and let~$X$ be a random $k$-vector with~$P=P^X$. 
Write~$V=O{\Lambda}O^\T$, where~$O$ is a $k\times k$ orthogonal matrix and~${\Lambda}={\rm diag}(\lambda_1,\ldots, \lambda_k)$ is a diagonal matrix with 
$\lambda_1\geq \lambda_2\geq \ldots\geq \lambda_k$. 
The affine invariance property from Theorem~\ref{Theoraffineinvariance} entails that
$$
D_{0}(V,P^X)
=
D_{0}\big(O^\T V O,P^{O^\T X}\big)
=
D_{0}\big(\Lambda,P^X\big)
.
$$
Denoting by~$e_1$ the first vector of the canonical basis of~$\R^{k^2}$, 
we then have
\begin{eqnarray}
\lefteqn{
\hspace{-13mm} 
D_0(V,P^X)
=
D_{0}\big(\Lambda,P^X\big)
\leq
{\rm pr}\big[e_1^\T {\rm vec}\big\{U_{0,\Lambda} U_{0,\Lambda}^\T- (1/k) I_k \big\} \geq 0 \big]
}
\nonumber
\\[2mm]
& &
\hspace{-7mm} 
= 
{\rm pr}\big[\{(U_{0,\Lambda})_1\}^2 \geq 1/k \big]
=
{\rm pr}\big[\{(U_{0,\Lambda})_1\}^2 \geq 1/k, X\neq 0 \big]
\nonumber
\\[2mm]
& & 
\hspace{-7mm} 
=
{\rm pr}\big\{ \lambda_1^{-1} X_1^2 \geq (1/k)
{\textstyle{\sum_{\ell=1}^k}} \lambda_\ell^{-1} X_\ell^2
, X\neq 0
 \big\}
 \leq  
{\rm pr}\big\{ X_1^2 \geq (1/k)
{\textstyle{\sum_{\ell=1}^k}} X_\ell^2
, X\neq 0
 \big\}
\\[2mm]
& & 
\hspace{-7mm} 
=
{\rm pr}\big( X_1^2/\|X\|^2 \geq 1/k
, X\neq 0
 \big)
=
{\rm pr}(X\neq 0)
{\rm pr}\big( 
U_1^2 \geq 1/k
 \big)
 ,
\label{ineqconst}
\end{eqnarray}
where $U=(U_1,\ldots,U_k)^\T$ is uniform over~$\mathcal{S}^{k-1}$. To have~$D_0(V,P^X)={\rm pr}(X\neq 0){\rm pr}\big(U_1^2 \geq 1/k \big)
$, the inequality in~(\ref{ineqconst}) needs to be an equality, which requires that~$\lambda_\ell=\lambda_1$ for all~$\ell$, hence that~$V=I_k$. 
\end{proof}



We can now prove Theorem~\ref{TheorFishconst}.

\begin{proof}[Proof of Theorem~\ref{TheorFishconst}]
Lemmas~A\ref{LemConsist1}-A\ref{LemConsist3} establish the result in the spherical case associated with~$\theta_0=0$ and~$V_0=I_k$. For general values of~$\theta_0$ and~$V_0$, note that $Y=V_0^{-1/2}(X-\theta_0)$ is elliptical with location~$0$, shape~$I_k$, and satisfies~${\rm pr}(Y=0)={\rm pr}(X=\theta_0)$. Writing
\begin{equation}
\label{asin55}	
W_0=\frac{k V_0^{-1/2}VV_0^{-1/2}}{{\rm tr}(V_0^{-1/2}VV_0^{-1/2})},
\end{equation}
affine invariance then entails that 
$$
D_{\theta_0}(V,P^X)
=
D_0(
W_0
,
P^Y
)
\leq D_0(I_k,P^Y)
=
D_0(
W_0
,
P^Y
)
=
D_{\theta_0}(V_0,P^X)
,
$$
with equality if and only if 
$W_0=I_k$, that is, if and only if $V=V_0$. 
\end{proof}
\vspace{0mm}


\begin{proof}[Proof of Theorem~\ref{Theoraffineinvariance}]
In the proof of Theorem~\ref{TheorQuasic}, we showed that
\begin{eqnarray*}
	\lefteqn{
D_\theta(V,P)
=
 \inf_{M\in\mathcal{M}_k^{\rm all}}
 {\rm pr}\big\{ 
 (X-\theta)^\T V^{-1/2} M V^{-1/2} (X-\theta) 
}
\\[0mm]
& & 
\hspace{53mm} 
 \geq  (1/k) {\rm tr}(M) 
 (X-\theta)^\T V^{-1} (X-\theta)
,X\neq \theta
 \big\}
 .
\end{eqnarray*}
Using the fact that 
$
V_A^{1/2}
=
k^{1/2} A V^{1/2} O/\{{\rm tr}(AV\! A^\T)\}^{1/2}
$
for some $k\times k$ orthogonal matrix~$O$, this readily yields
\begin{eqnarray*}
	\lefteqn{
\hspace{-3mm} 
D_{A\theta+b}(V_A,P)
=
 \inf_{M\in\mathcal{M}_k^{\rm all}}
{\rm pr}\big\{ 
 (X-\theta)^\T  
V^{-1/2} O
  M 
  O^\T V^{-1/2}  (X-\theta) 
}
\\[0mm]
& & 
\hspace{36mm} 
 \geq  (1/k) {\rm tr}(OMO^\T) 
  (X-\theta)^\T V^{-1} (X-\theta)
,X\neq \theta
 \big\}
=
D_{\theta}(V,P)
,
\end{eqnarray*}
as was to be shown.
The affine-equivariance property of the depth regions readily follows.  
\end{proof}
\vspace{0mm}


The proof of Theorem~\ref{TheorDistribfreeness} requires the following lemma, whose proof is straightforward, hence is omitted.

\begin{lemma}
\label{lemineq}
For any~$v_1,v_2$ such that~$v_1^2+v_2^2<1$, we have
$$
\frac{(1-v_1^2)^{1/2}-|v_2|}{(1-v_1^2)^{1/2}+|v_2|}
\leq
\frac{(1-v_1^2)^{1/2}+|v_2|}{(1-v_1^2)^{1/2}-|v_2|}
\leq
\frac{1+(v_1^2+v_2^2)^{1/2}}{1-(v_1^2+v_2^2)^{1/2}}
\cdot
$$
\end{lemma}
\vspace{3mm}

%

\vspace{-3mm}

\begin{proof}[Proof of Theorem~\ref{TheorDistribfreeness}]
(i) 
If~$P=P^X$ is elliptical with location~$\theta_0$ and shape~$V_0$, then $V_0^{-1/2}(X-\theta_0)$ is equal in distribution to~$RU$, where~$U$ is uniformly distributed over the unit sphere~$\mathcal{S}^{k-1}$ and is independent of the nonnegative random variable~$R$. Theorem~\ref{Theoraffineinvariance} 
then yields
\begin{equation}
\label{fsgt}	
D_{\theta_0}(V,P^X)
=
D_0(
W_0
,
P^{RU}
)
,
\end{equation}
where~$W_0$ is as in~(\ref{asin55}). Now, for any~$\tilde{V}\in\mathcal{P}_{k,{\rm tr}}$, Lemma~\ref{LemrewritingD} entails that
\begin{eqnarray}
D_0(\tilde{V},P^{RU})
&=&
 \inf_{M\in\mathcal{M}_k^0}
{\rm pr}\big( 
U^\T \tilde{V}^{-1/2} M \tilde{V}^{-1/2} U
 \geq 
 0
 , R> 0
 \big)
\label{inftocomputepre}	
\\[2mm]
&=&
{\rm pr}(R> 0)
\inf_{M\in\mathcal{M}_k^0}
{\rm pr}\big( 
U^\T \tilde{V}^{-1/2} M \tilde{V}^{-1/2} U
 \geq 
 0
 \big)
=
{\rm pr}(R> 0)
D_0(\tilde{V},P^U)
.
\nonumber
\end{eqnarray}
Combining with~(\ref{fsgt}), we obtain
$$
D_{\theta_0}(V,P^X)
=
(1-{\rm pr}_{P^X}[\{\theta_0\}])
D_0(
W_0
,
P^{U}
)
,
$$
which establishes Part~(i) of the result. 
(ii) Assume that~$P=P^X$ is bivariate standard normal
and fix~\mbox{$V\in\mathcal{P}_{2,{\rm tr}}$}.  We aim at evaluating
\begin{equation}
	\label{inftocomputepre25}	
D_{0}(V,P^X)
=
 \inf_{M\in\mathcal{M}_k^0}
{\rm pr}\big( 
X^\T V^{-1/2} M V^{-1/2} X
 \geq 
 0
 \big)
;
\end{equation}
see~(\ref{inftocomputepre}). To do so, it will be convenient to parametrise~$V$ and the matrix~$M$ as
$$
V
=
\bigg(
\begin{array}{cc}
1+v_1 & v_2 \\[1mm]
v_2 & 1-v_1	
\end{array}
\bigg)
\quad
\textrm{ and }
\quad
M
=
m_1
\bigg(
\begin{array}{cc}
1 & m_2 \\[1mm]
m_2 & -1	
\end{array}
\bigg)
,
$$
with~$v_1^2+v_2^2<1$ and~$m_1\neq 0$. Indeed,~$m_1=0$ makes the probability in~(\ref{inftocomputepre25}) equal to one, which cannot be the infimum. 
Decomposing~$V^{-1/2}MV^{-1/2}$ into~$O\Lambda O^\T$, where~$O$ is a $2\times 2$ orthogonal matrix and where~$\Lambda={\rm diag}\{\lambda_1(V^{-1}M),\lambda_2(V^{-1}M)\}$, with~$\lambda_1(V^{-1}M)\geq \lambda_2(V^{-1}M)$, involves the eigenvalues of~$V^{-1}M$ or, equivalently, of~$V^{-1/2}MV^{-1/2}$, we have
\begin{equation}
	\label{inftocompute}
D_0(V,P)
=
 \inf_{(m_1,m_2)\in\R_0\times \R}
{\rm pr}\big\{ 
\lambda_1(V^{-1}M) X_1^2 +\lambda_2(V^{-1}M) X_2^2  \geq  0
 \big\}
,
\end{equation}
where~$X=(X_1,X_2)^\T$ is still bivariate standard normal. 
%
%
%
%
%
%
%
Since~$\lambda_1(-V^{-1}M)=-\lambda_2(V^{-1}M)$ for any~$M\in\mathcal{M}_k^0$, we have
\begin{eqnarray}
\lefteqn{
D_0(V,P)
=
\min
\bigg[
 \inf_{(m_1,m_2)\in\R^+_0\times \R}
{\rm pr}\big\{ 
\lambda_1(V^{-1}M) X_1^2 +\lambda_2(V^{-1}M) X_2^2  \geq  0
 \big\}
,
}
	\label{inftocompute2}
\\
& & 
\hspace{38mm} 
 \inf_{(m_1,m_2)\in\R^+_0\times \R}
{\rm pr}\big\{ 
\lambda_1(V^{-1}M) X_1^2 +\lambda_2(V^{-1}M) X_2^2  \leq  0
 \big\}
\bigg]
,
\nonumber
\end{eqnarray}
which allows us to restrict to positive values of~$m_1$. We will show below that~$\lambda_2(V^{-1}M)<0<\lambda_1(V^{-1}M)$ for any~$M\in\mathcal{M}_k^0$. 
A direct computation shows that, for~$m_1>0$, 
$$
\lambda_1(V^{-1}M)
=
\frac{m_1}{\det V}
\,
\big[
-(v_1+m_2v_2)+\{(v_1+m_2v_2)^2+(1+m_2^2)\det V\}^{1/2}
\,\big]
>
0
$$
and
$$
\lambda_2(V^{-1}M)
=
\frac{m_1}{\det V}
\,
\big[
-(v_1+m_2v_2)-\{(v_1+m_2v_2)^2+(1+m_2^2)\det V\}^{1/2}
\,\big]
<
0
.
$$
Since~$f(m_2)=-\lambda_2(V^{-1}M)/\lambda_1(V^{-1}M)$ does not depend on~$m_1$, (\ref{inftocompute2}) leads to
\begin{eqnarray}
D_0(V,P)
& = &
\min
\bigg[
{\rm pr}\bigg\{ 
 X_1^2/X_2^2   \geq   \sup_{m_2\in\R} f(m_2)
 \bigg\}
,
{\rm pr}\bigg\{ 
 X_1^2/X_2^2   \leq   \inf_{m_2\in\R} f(m_2)
 \bigg\}
\bigg]
\nonumber
\\[2mm]
& = &
{\rm pr}\bigg[ 
 X_1^2/X_2^2   \geq  
\max\bigg\{
 \sup_{m_2\in\R} f(m_2)
 \ 
 ,
 \ 
 1\,/ \inf_{m_2\in\R} f(m_2)
 \bigg\}
 \bigg]
.
\label{keyy}
\end{eqnarray}
It is easy to check that~$f$ is differentiable over~$\R$ with a derivative of the form~$c_{v_1,v_2}(m_2) \linebreak (v_2-v_1 m_2)$, where~$c_{v_1,v_2}(m_2)>0$ for any~$m_2$, and that
$$
f(\pm\infty)
=
\lim_{m_2\to \pm\infty} f(m_2)
=
\frac{(1-v_1^2)^{1/2}\pm v_2}{(1-v_1^2)^{1/2} \mp v_2}
\cdot
$$
We treat the cases~$v_1=0$ and~$v_1\neq 0$ separately.
\vspace{3mm}
 
 (a) Assume that~$v_1= 0$. If~$v_2=0$, then~$V=I_2$ and Theorem~\ref{TheorFishconst} establishes the result. If~$v_2\neq 0$, then~$f$ has no critical point and 
$$
 \sup_{m_2\in\R}
f(m_2)
=
\max\big\{
f(-\infty)
,
f(\infty)
\big\}
 =
\frac{1+|v_2|}{1-|v_2|}
$$
and
$$ 
\inf_{m_2\in\R}
f(m_2)
=
\min\big\{
f(-\infty)
,
f(\infty)
\big\}
=
\frac{1-|v_2|}{1+|v_2|}
,
$$
so that~(\ref{keyy}) yields
\begin{eqnarray*}
\lefteqn{
D_0(V,P)
=
{\rm pr}\Bigg( 
 \frac{X_1^2}{X_2^2}   \geq  
\frac{1+|v_2|}{1-|v_2|}
 \Bigg)
}
\\[2mm]
& &
\hspace{23mm} 
=
{\rm pr}\Bigg\{  
 \frac{X_1^2}{X_2^2}   \geq  
\frac{1+(1-\det V)^{1/2}}{1-(1-\det V)^{1/2}}
 \Bigg\}
=
{\rm pr}\bigg\{
Y_2
\geq
\frac{1}{2}
+
\frac{1}{2}
\big(
1- \det V 
\big)^{1/2} 
\bigg\}
 ,
\end{eqnarray*}
where we have used the fact that if~$Z$ has a $F(1,1)$ Fisher-Snedecor distribution, then~$Z/(1+Z)$ has a ${\rm Beta}(1/2,1/2)$ distribution. 
\vspace{3mm}
 
(b) Assume now that~$v_1\neq 0$. Then the only critical point of~$f$ is~$m_2^{\rm crit}=v_2/v_1$, so that, irrespective of the fact that this critical point is a local minimum/maximum of~$f$, 
\begin{eqnarray*}
 \sup_{m_2\in\R}
f(m_2)
&=&
\max\big\{
f(-\infty)
,
f(\infty)
,
f(m_2^{\rm crit})
\big\}
\\[2mm]
& =&
\max\Bigg\{
\frac{(1-v_1^2)^{1/2}+|v_2|}{(1-v_1^2)^{1/2}-|v_2|}
,
\frac{{\rm sign}(v_1)+(v_1^2+v_2^2)^{1/2}}{{\rm sign}(v_1)-(v_1^2+v_2^2)^{1/2}}
 \Bigg\}
\end{eqnarray*}
and
\begin{eqnarray*}
 \inf_{m_2\in\R}
f(m_2)
& =&
\min\big\{
f(-\infty)
,
f(\infty)
,
f(m_2^{\rm crit})
\big\}
\\[2mm]
& =&
\min\Bigg\{
\frac{(1-v_1^2)^{1/2}-|v_2|}{(1-v_1^2)^{1/2}+|v_2|}
,
\frac{{\rm sign}(v_1)+(v_1^2+v_2^2)^{1/2}}{{\rm sign}(v_1)-(v_1^2+v_2^2)^{1/2}}
 \Bigg\}
.
\end{eqnarray*}
Lemma~\ref{lemineq} yields
$$
 \sup_{m_2\in\R}
f(m_2)
=
\frac{(1-v_1^2)^{1/2}+|v_2|}{(1-v_1^2)^{1/2}-|v_2|}
\,
\mathbb{I}(v_1<0)
+
\frac{1+(v_1^2+v_2^2)^{1/2}}{1-(v_1^2+v_2^2)^{1/2}}
\,
\mathbb{I}(v_1>0)
$$
and
$$
 \inf_{m_2\in\R}
f(m_2)
=
\frac{-1+(v_1^2+v_2^2)^{1/2}}{-1-(v_1^2+v_2^2)^{1/2}}
\,
\mathbb{I}(v_1<0)
+
\frac{(1-v_1^2)^{1/2}-|v_2|}{(1-v_1^2)^{1/2}+|v_2|}
\,
\mathbb{I}(v_1>0)
,
$$
hence also
$$
\max\bigg\{
 \sup_{m_2\in\R} f(m_2)
 \ 
 ,
 \ 
 1\,/ \inf_{m_2\in\R} f(m_2)
 \bigg\}
=
\frac{1+(v_1^2+v_2^2)^{1/2}}{1-(v_1^2+v_2^2)^{1/2}}
=
\frac{1+(1-\det V)^{1/2}}{1-(1-\det V)^{1/2}}
\cdot
$$
Therefore, (\ref{keyy}) finally provides
$$
D_0(V,P)
=
{\rm pr}\Bigg\{  
 \frac{X_1^2}{X_2^2}   \geq  
\frac{1+(1-\det V)^{1/2}}{1-(1-\det V)^{1/2}}
 \Bigg\}
=
{\rm pr}\bigg\{
Y_2
\geq
\frac{1}{2}
+
\frac{1}{2}
\big(
1- \det V 
\big)^{1/2} 
\bigg\}
.
$$
This proves the result for the case where~$P$ is bivariate standard normal. The general result then follows from Part~(i) of the theorem. 
\end{proof}
\vspace{3mm}


%
%
%
%
%
%

\begin{proof}[Proof of Theorem~\ref{Theorunifconst}]
(i) Let~$P$ and~$Q$ be two probability measures over~$\R^k$
and fix~$V\in\mathcal{P}_{k,{\rm tr}}$. Fix~$\varepsilon>0$ and assume, without loss of generality, that~$D_\theta(V,P)\leq D_\theta(V,Q)$. Lemma~\ref{LemrewritingD} entails that there exists~$M_0\in\mathcal{M}^0_k$ such that
$
{\rm pr}_P\big( C^{M_0}_{\theta,V} \big)
\leq 
D_\theta(V,P) + \varepsilon
,      
$
where we still use the notation~$C^M_{\theta,V}
 =
 \big\{ x \in \R^{k}\setminus\{\theta\} : 
 (u_{\theta,V}^x)^\T M u_{\theta,V}^x 
 \geq  {\rm tr}(M)/k   
 \big\}$. 
Consequently, using Lemma~\ref{LemrewritingD} again,
\begin{eqnarray*}
\lefteqn{
\hspace{-0mm} 
|D_\theta(V,Q) - D_\theta(V,P)|
=
D_\theta(V,Q) - D_\theta(V,P)
}
\\[2mm]
& & 
\hspace{25mm} 
\leq 
{\rm pr}_Q\big( C^{M_0}_{\theta,V} \big)
-
{\rm pr}_P\big( C^{M_0}_{\theta,V} \big)
+
\varepsilon
\leq
\sup_{C\in \mathcal{C}_\theta} | {\rm pr}_Q(C)-{\rm pr}_P(C) | 
+ 
\varepsilon
,
\end{eqnarray*}
with~$\mathcal{C}_\theta=\{C^M_{\theta,V} : M\in\mathcal{M}^0_k, V\in\mathcal{P}_{k,{\rm tr}}\}$. Since this holds for any~$\varepsilon>0$ and for any~$V\in\mathcal{P}_{k,{\rm tr}}$, we have  
$$
\sup_{V\in\mathcal{P}_{k,{\rm tr}}} 
|D_\theta(V,Q) - D_\theta(V,P)|
\leq
\sup_{C\in \mathcal{C}_\theta} | {\rm pr}_Q(C)-{\rm pr}_P(C) | 
.
$$
It thus only remains to show that $\mathcal{C}_\theta$ is a Vapnik-Chervonenkis class. To do so, note that~$C^M_{\theta,V}
 =
 \big\{ x \in \R^{k}\setminus\{\theta\} : (x-\theta)^\T V^{-1/2} M V^{-1/2} (x-\theta)  \geq  0    \big\}
 $, so that~$\mathcal{C}_\theta\subset \{D_{\theta,A}\cap (\R^{k}\setminus\{\theta\}) : A\in \mathcal{M}_k^{\rm all} \}$, with
 $
D_{\theta,A} 
=
 \big\{ x \in \R^{k}: (x-\theta)^\T A (x-\theta)  \geq  0 \big\}
$.  Theorem~4.6 from \cite{Dud2014} implies that~$\{D_{\theta,A}: A\in \mathcal{M}_k^{\rm all} \}$ is a Vapnik-Chervonenkis class $\mathcal{D}_\theta$. It then follows from Lemma~2.6.17(ii) in \cite{WelVan1996} that~$\{D_{\theta,A}\cap (\R^{k}\setminus\{\theta\}) : A\in \mathcal{M}_k^{\rm all} \}$, hence also~$\mathcal{C}_\theta$, is a Vapnik-Chervonenkis class. 
(ii) The proof is long and technical, but follows along the same lines as the proof of Theorem~2.2 in \cite{PVB17}, hence is omitted for the sake of brevity. 
\end{proof}
   \vspace{0mm}

\begin{proof}[Proof of Theorem~\ref{Theormaxdepthconst}]
(i) Recall from~(\ref{barydef}) that~$V_{\theta,P}$ is defined as the barycentre of $R_\theta(\alpha_*,P)$, with~$\alpha_*=\max_V D_\theta(V,P)$. The mapping~$V\mapsto D_\theta(V,P)$ is upper semicontinuous (Theorem~\ref{Theorcontinuity}) and constant over~$R_\theta(\alpha_*,P)$. Clearly, it is easy to define a mapping~$V\mapsto \tilde{D}_\theta(V,P)$ that is upper semicontinuous, agrees with~$V\mapsto D_\theta(V,P)$ in the complement of~$R_\theta(\alpha_*,P)$, and for which~$V_{\theta,P}$ is the unique maximizer. By using Theorem~\ref{Theorunifconst}, it follows from Theorem~2.12 and Lemma~14.3 in \cite{Kos2008} that $d(V_{\theta,P_n},V_{\theta,P}) \to 0$ almost surely as~$n\to\infty$. Part~(i) of the result then follows from the fact that, in neighbourhoods of the form~$\{V: d(V,V_{\theta,P})<\varepsilon\}$, there exists a constant~$C=C_\varepsilon$ such that~$d_F(V,V_{\theta,P})<C d(V,V_{\theta,P})$, where~$d_F$ is the Frobenius distance.
(ii) The proof is entirely similar, hence is omitted. 
\end{proof} 
   \vspace{0mm}

\begin{proof}[Proof of Theorem~\ref{dimD}]
Let $L_{\theta,V}=U_{\theta,V} U^\T_{\theta,V} - (1/k) I_k$. Since
$(L_{\theta,V})_{11}=-\sum_{\ell=2}^k (L_{\theta,V})_{\ell\ell}$, 
there exists a $(d_k+1)\times d_k$ full-rank matrix~$H_0$ such that
$
{\rm vech}( L_{\theta,V})
=
H_0\, {\rm vech}_0( L_{\theta,V})
.
$ 
Therefore, there exists a $k^2\times d_k$ full-rank matrix~$H$ such that
$
W_{\theta,V}
=
{\rm vec}( L_{\theta,V})
=
H
\tilde{W}_{\theta,V}
$.
One can, for example, take~$H=DH_0$, where~$D$ is the usual duplication matrix.  
It follows that
\begin{eqnarray*}
\lefteqn{
D_\theta(V,P)
=
D(0,P^{W_{\theta,V}})
=
\inf_{u\in\R^{k^2}} {\rm pr}\big(u^\T W_{\theta,V} \geq 0\big)
}
\\[2mm]
& & 
\hspace{13mm} 
=
\inf_{u\in\R^{k^2}} {\rm pr}\big\{(H^\T u)^\T \,\tilde{W}_{\theta,V} \geq 0\big\}
=
\inf_{v\in\R^{d_k}} {\rm pr}\big(v^\T \tilde{W}_{\theta,V} \geq 0\big)
=
D(0,P^{\tilde{W}_{\theta,V}})
,
\end{eqnarray*}
where we used the fact that~$H^\T$ has full column rank.
\end{proof}


\bibliographystyle{rss}
\bibliography{Paper.bib}


\begin{figure}[h!]
\begin{center}
\includegraphics[width=\textwidth]{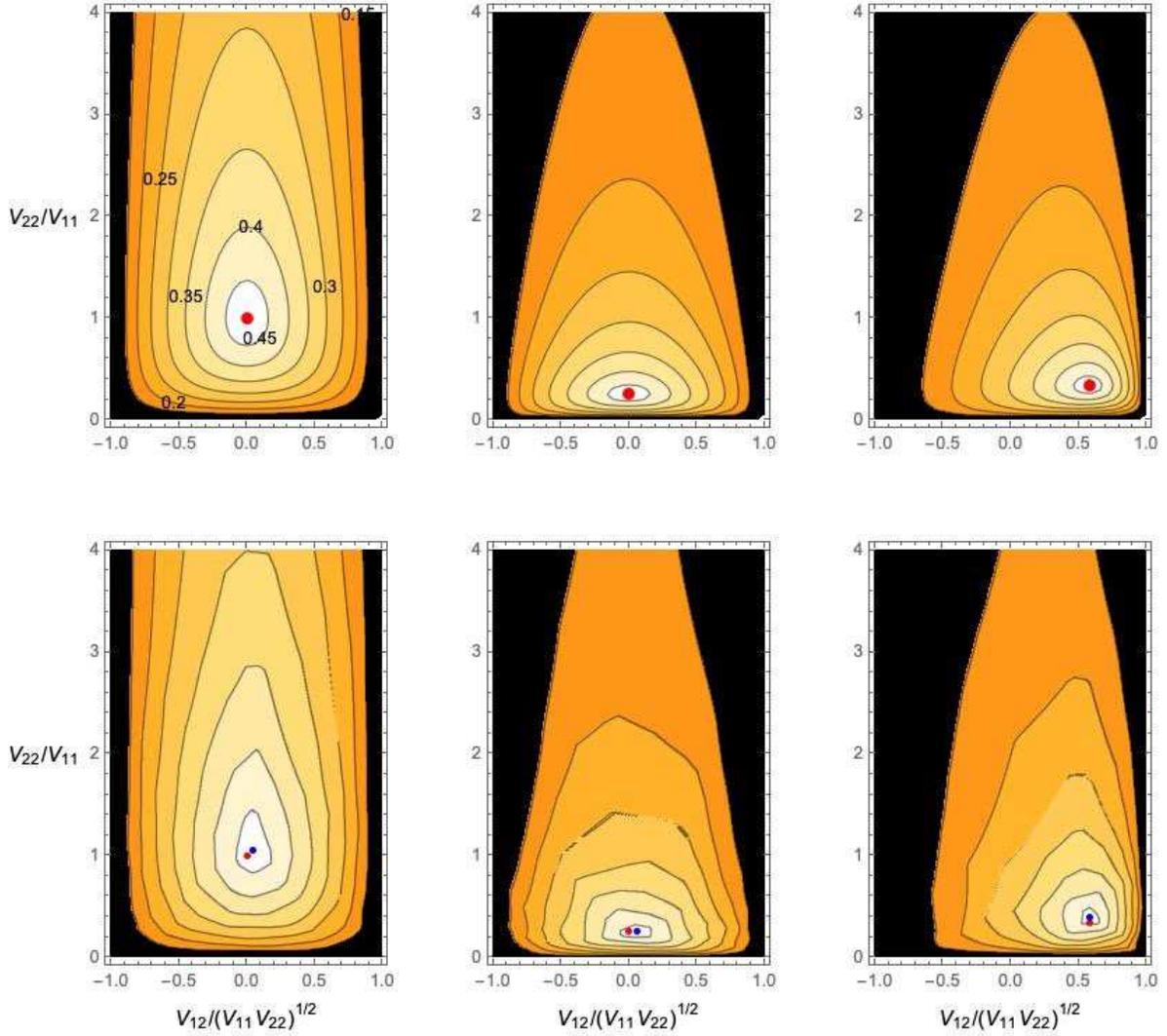}
\vspace{-6mm}
\end{center}
\caption{
(First row:) Contour plots of~$D_\theta(V,P)$ in terms of~$V_{12}/(V_{11}V_{22})^{1/2}$ and~$V_{22}/V_{11}$, where~$P$ refers to an arbitrary bivariate elliptical probability measure with location~$\theta$ and with shape~$V_A={\rm diag}(1,1)$ (left), $V_B= {\rm diag}(1.6,0.4)$ (center), or $V_C$ with diagonal vector $(1.5,0.5)^T$ and off-diagonal elements $0.5$ (right), 
\vspace{.3mm}
 that is so that~${\rm pr}[\{\theta\}]=0$. (Second row:) The corresponding contour plots of~$D_0(V,P_n)$, where~$P_n$ is the empirical probability measure associated with a random sample of size~$n=800$ from the centered bivariate normal with shape~$V_A$ (left), $V_B$ (center), and $V_C$ (right). The true shapes~$V_{0,P}$ and sample deepest shapes~$V_{0,P_n}$ are marked in red and in blue, respectively.} 
\label{Fig1}  
\end{figure}


\begin{figure}[h]
\includegraphics[width=\textwidth]{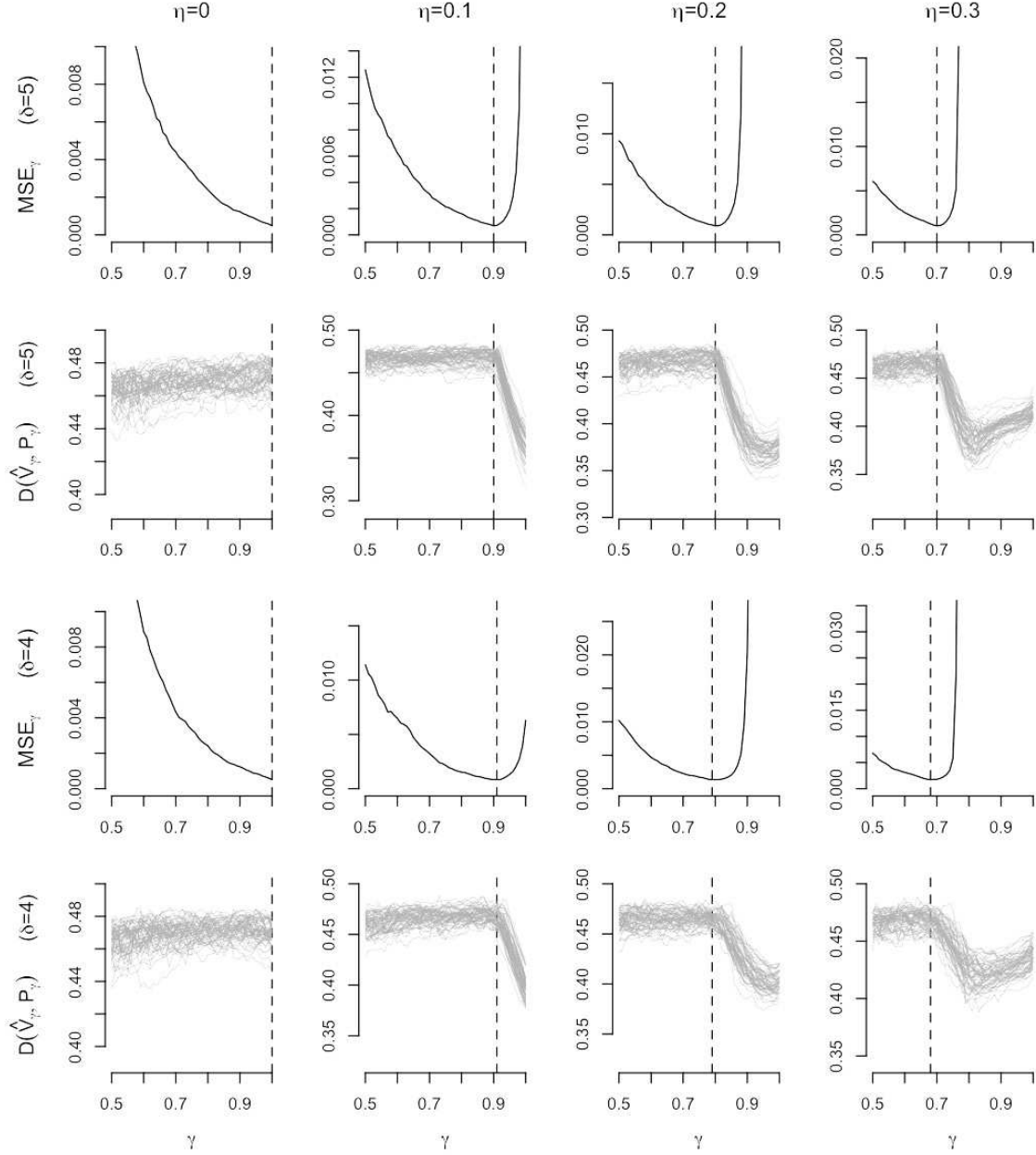}
\vspace{-2mm}
\caption{
(First row:) Plot of the mapping~$\gamma\mapsto \textrm{MSE}_\gamma$ in the easy case~$\delta=5$
\vspace{-.7mm}
  and for contamination proportions~$\eta\in\{0,0.1,0.2,0.3\}$. 
(Second row). Plot of $50$ random curves~$\mathcal{C}=\{(\gamma, D(\hat{V}_\gamma,P_{\gamma,n})):\gamma\in[0.5,1]\}$, still for~$\delta=5$ and for the same contamination proportions.  
(Third and fourth rows:) The corresponding plots for the harder case associated with~$\delta=4$; see $\S$~\ref{Secapplis} for details. Each panel shows a vertical line at~$\gamma_0=\arg\min_\gamma \textrm{MSE}_\gamma$.
} 
\label{Fig2} 
\end{figure}


\begin{figure}[h]
\includegraphics[width=\textwidth]{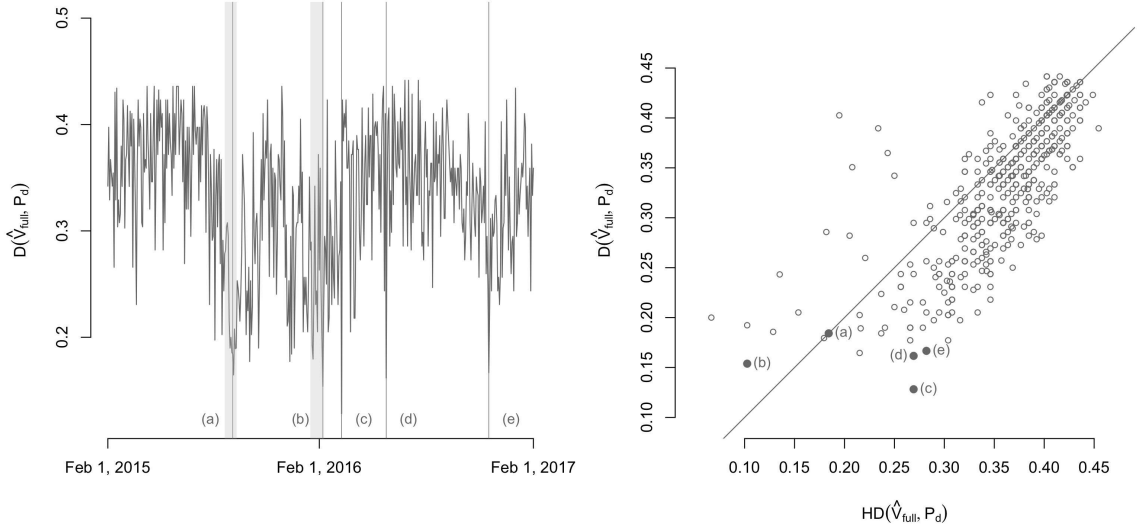}
\caption{
(Left:) Plot of $D(\hat V_{\rm full},P_d)$ as a function of $d$. Events (a) to (e) are described in $\S$~\ref{SecAppliOutlier}. 
(Right:) Plot of $D(\hat V_{\rm full},P_{d})$ vs $H\!D(\hat V_{\rm full},P_d)$ for each trading day~$d$. Events from the left panel are highlighted using the same colour.
} 
\label{Fig3}  
\end{figure}


\begin{figure}[h!]
\begin{center} 
\includegraphics[width=\textwidth]{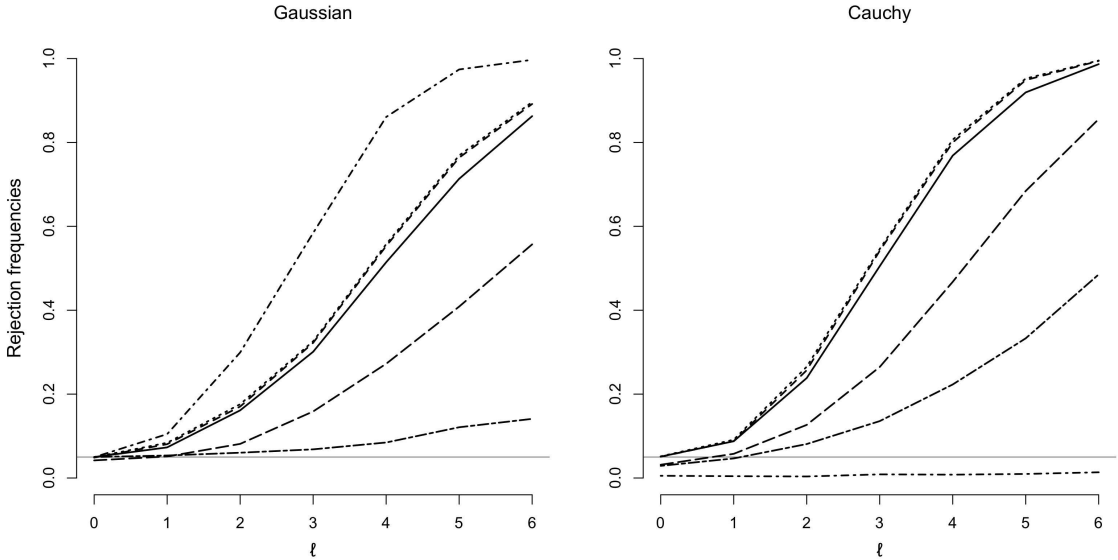}
\vspace{-6mm} 
\end{center}
\caption{
Rejection frequencies, under bivariate normal (left) and elliptical Cauchy (right) densities, of six tests of sphericity: the Gaussian test (dot-dashed curve), the sign test (dotted curve), the test based on Tyler's scatter matrix (dashed curve), the depth-based test (solid curve), and two minimum covariance determinant-based tests based on different trimming proportions (long dashed and short-long dashed curves for trimming proportion~$0.2$ and~$0.5$, respectively). Results are based on $3,\!000$ replications and the sample size is $n=500$. See $\S$~\ref{Sectest} for details.
} 
\label{Fig4}   
\end{figure}


\begin{figure}[h]
\includegraphics[width=\textwidth]{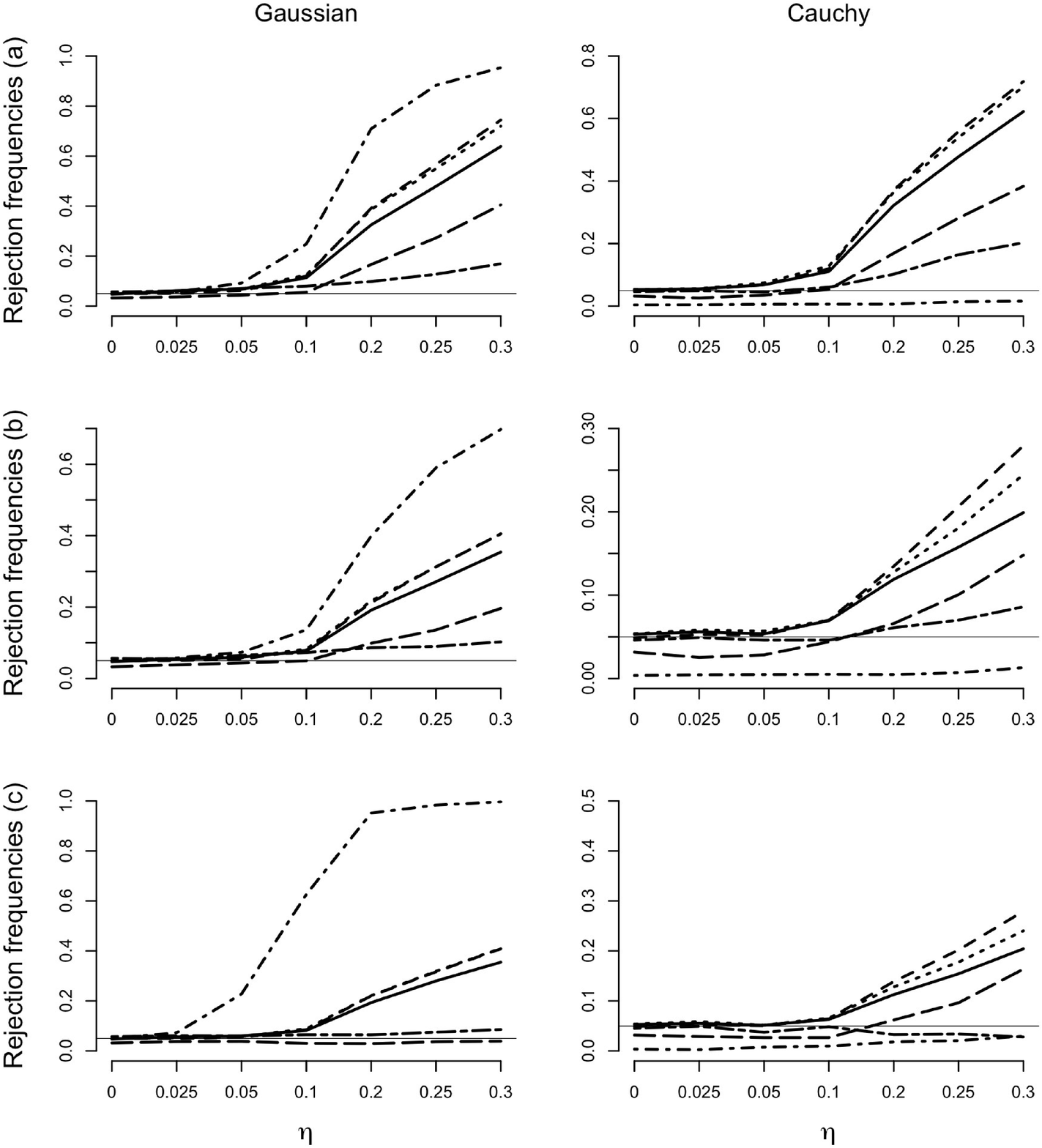}
\caption{
Null rejection frequencies, as a function of the contamination level~$\eta$, of the same six tests (using the same line types) as in Figure~\ref{Fig4}, under bivariate normal (left) and elliptical Cauchy (right) densities. The labels~(a)--(c) refer to the three contaminations patterns considered; see $\S$~\ref{Sectest} for details. Results are based on $3,\!000$ replications and the sample size is $n=200$. 
} 
\label{Fig5}  
\end{figure}


\end{document}